\documentclass[]{interact}
\usepackage{epstopdf}
\usepackage[caption=false]{subfig}

\usepackage[numbers,sort&compress]{natbib}
\bibpunct[, ]{[}{]}{,}{n}{,}{,}
\renewcommand\bibfont{\fontsize{10}{12}\selectfont}

\newcommand{\norm}[1]{{\left\|{#1}\right\|}}   \newcommand{\scal}[1]{{\left\langle{#1}\right\rangle}}
\newcommand{\Z}{\mathbb Z} 
  \newcommand{\bw}{\overline{w}}
\theoremstyle{plain}
\newtheorem{theorem}{Theorem}[section]
\newtheorem{lemma}[theorem]{Lemma}
\newtheorem{corollary}[theorem]{Corollary}
\newtheorem{proposition}[theorem]{Proposition}

\theoremstyle{definition}
\newtheorem{definition}[theorem]{Definition}

\theoremstyle{remark}
\newtheorem{remark}[theorem]{Remark}

\usepackage[colorlinks=true,linkcolor=red]{hyperref}%
\usepackage[usenames,dvipsnames]{color}
 
\newcommand{\C}{\mathbb C} 

\newcommand{\R}{\mathbb R}
\newcommand{\D}{\mathbb{D}}

\newcommand{\bz}{\overline{z}} 
\numberwithin{equation}{section}
\makeatletter\@addtoreset{equation}{section}
\usepackage{color}
\usepackage{colortbl}
\begin{document}
	
	
	\title{Fractional Zernike functions}
	
	\author{
		\name{Hajar Dkhissi 
			\textsuperscript{a},   Allal Ghanmi	\textsuperscript{a} and  Safa Snoun\textsuperscript{b}\thanks{CONTACT A. Ghanmi. Emails: allal.ghanmi@um5.ac.ma/ag@fsr.ac.ma}}
		\affil{\textsuperscript{a}	
			Analysis, P.D.E $\&$ Spectral Geometry, Lab M.I.A.-S.I., CeReMAR, \newline
			Department of Mathematics,
			P.O. Box 1014,  Faculty of Sciences, \newline
			Mohammed V University in Rabat, Morocco}
		\affil{\textsuperscript{b}  
			Department  of Mathematics, Faculty of Sciences, \newline University of Gabes 6072, Tunisia}
	}

	\maketitle
	
	\begin{abstract}
		We consider and provide an accurate study for the fractional Zernike functions on the punctured unit disc, generalizing the classical Zernike polynomials and their associated $\beta$-restricted Zernike functions. Mainly, we give the spectral realization of the latter ones  and show that they are orthogonal $L^2$-eigenfunctions for certain perturbed magnetic (hyperbolic) Laplacian.		
		The algebraic and analytic properties for the fractional Zernike functions to be established include the connection to special functions, their zeros, their orthogonality property,  as well as the differential equations, recurrence and operational formulas they satisfy. Integral representations are also obtained. Their regularity as poly-meromorphic functions is discussed and their generating functions including a bilinear one of "Hardy--Hille type" are derived.  Moreover, we prove that a truncated subclass defines a complete orthogonal system in the underlying Hilbert space giving rise to a specific Hilbertian orthogonal decomposition in terms of a second class of generalized Bergman spaces.  	
	\end{abstract}
	


\tableofcontents

\section{Introduction}

The classical real Zernike polynomials 
are introduced  in the framework of optical problems,  especially in order to analyze the figure of a circular mirror. In Zernike"s paper on the knife edge test and the phase contrast method \cite{Zernike34}, they are defined as eigenfunctions of a rotational invariant second order
partial differential equation.
Next, they have been used in the Nijboer's works to develop the diffraction theory of optical aberrations. Since then, they have been extensively employed to express the propagation of a wavefront data in optical tests through imaging system \cite{GrayDunn2012,IskanderCollinsDavis2001,LakshminarayananFleck2011},  and to represent the aberrations of optical systems (by atmospheric turbulence) \cite{Noll1976,Winker1991}. They are also used to study diffraction problems in the rotationally symmetric system with circular pupils \cite{Nijboer1947,ZernikeBrinkman35} and pattern recognition \cite{KhotanzadHong1990,WangHealey1998}. 
More recently, they are applied efficiently 
to characterize the shape of any portion of molecular surfaces and to evaluate the shape complementarity of protein-protein interfaces \cite{MilanettiMiottoDiRienzoMontiGostiRuocco2021}.

A generalized complex version (called Zernike or disc polynomials) defines them 
as the orthogonal ones on the unit disc $\mathbb{D}=\{z\in \C; \, |z|<1\}$ with finite values at the boundary. They are given by the Rodrigues type formula
\begin{align}\label{DiscePol}
	\mathcal{Z}^{\gamma}_{m,n}(z,\bz) 
	&:=  		(-1)^{m+n}   (1-|z|^2)^{-\gamma} 
	\frac{\partial^{m+n}}{\partial z^m \bz^n} \left( 1-|z|^2\right)^{\gamma+m+n}  
\end{align}
for varying nonnegative integers $m,n$, and real $\gamma>-1$.
This definition agrees with the one provided by Koornwinder \cite{Koornwinder75
} as well as the one considered by Dunkl \cite{
	Dunkl84}.
%
Algebraic and analytic properties of $ \displaystyle \mathcal{Z}_{m,n}^\gamma (z,\bar z)$ have been discussed in many papers \cite{Aharmim2015,Koornwinder75,
	Wunsch05}.
The corresponding Wiener and Paley type theorems have been obtained by Kanjin in \cite{Kanjin2013}.
Recently, they have shown to be useful in the concrete description of spectral properties of different types of Cauchy transforms  \cite{
	ElHElkGh21,ElHGhIn2015Arxiv}.  

In the present paper, we consider a specific generalization of the Zernike polynomials in \eqref{DiscePol}. Namely, we deal with the family of functions 
\begin{align}\label{betZpol}
	\mathcal{Z}^{\kappa,\rho}_{m,n}(z,\bz) 
	&:=  		(-1)^m z^{-\rho} (1-|z|^2)^{-\kappa} 
	\frac{\partial^m}{\partial z^m}  \left(  z^{n+\rho} (1-|z|^2)^{\kappa+m}\right) 
\end{align}
on the punctured unit disc $\mathbb{D}^*=\mathbb{D} \setminus\{0\}$, 
for fixed real numbers  $\rho,\kappa>-1$ and varying integers $m$ and $n$ such that $m\geq 0$ and  $n+\rho \geq 0$.
Thus, for arbitrary nonnegative integer $\rho$, 
they reduce further to the Zernike polynomials in \eqref{DiscePol} since  for every $\ell =0,1, \cdots$
we have
\begin{align}\label{frZcZ} z^{\ell}\mathcal{Z}^{\kappa,\ell}_{m,n}(z,\bz)&=   \frac{\mathcal{Z}_{m,n+\ell}^\kappa (z,\bz)}{(\kappa+m+1)_{n+\ell}}  .
\end{align}
Otherwise,
they are no longer polynomials. Their study for arbitrary $\rho$ can be reduced to the subclass  corresponding to $0 \leq \rho <1$.
More precisely, we have
$$ \mathcal{Z}^{\kappa,\rho}_{m,n}(z,\bz) = 
z^{-[\rho]} 	\mathcal{Z}^{\kappa,\widetilde{\rho}}_{m,n+ [\rho]}(z,\bz) 
$$
whenever $ n+[\rho]\geq 0 $, where $[\rho]$ denotes the integer part of $\rho$ and
$0 \leq \widetilde{\rho}=\rho-  [\rho] <1$. 
This justifies somehow the following definition which can also be justified from being poly-meromorphic (see Theorem \ref{thmAnaMeremor}).

\begin{definition} The functions  $\mathcal{Z}^{\kappa,\rho}_{m,n}(z,\bz)$ in \eqref{betZpol} are referred to as fractional Zernike functions.
\end{definition}

Contrary to the classical Zernike polynomials 
satisfying  the symmetry relationships $\overline{\mathcal{Z}^{\kappa}_{m,n}(z,\bz)}  = \mathcal{Z}^{\kappa }_{m,n}(\bz,z) = 	\mathcal{Z}^{\kappa }_{n,m}(z,\bz),$
which play a crucial rule in their study, this relation is no longer valid for the fractional Zernike functions $\mathcal{Z}^{\kappa,\rho}_{m,n}(z,\bz)$ even if $\rho$ is a positive integer. In fact, we have only $\overline{\mathcal{Z}^{\kappa,\rho}_{m,n}(z,\bz)}= \mathcal{Z}^{\kappa,\rho}_{m,n}(\bz,z) 
$ for arbitrary $\rho$ and one gets from \eqref{frZcZ} the identity
\begin{align}\label{Sym}
	\overline{\mathcal{Z}^{\kappa,\rho}_{m,n}(z,\bz)}
	=   \frac{(\kappa+1)_m}{(\kappa+1)_{n+\rho}} z^\rho \bz^{-\rho}  \mathcal{Z}^{\kappa,\rho}_{n+\rho,m-\rho}(z,\bz)
\end{align} 
valid for $\rho$ being a nonnegative integer.
This reveals in particular that the  analytic and spectral properties of the functions $\mathcal{Z}^{\kappa,\rho}_{m,n}(z,\bz)$ 
can not be directly recovered from the Zernike polynomials, and the relevant properties
may be completely different from the classical ones, essentially when  $\rho$ is non integer. Thus, a concrete description of their algebraic and analytic properties  for fixed reals $\rho,\kappa>-1$ is desirable. 

To this purpose  we begin by considering the so-called $\beta$-restricted Zernike functions  $\psi^{\gamma,\eta}_{m,n}$. They are shown to be a special class of polyanalytic excited states in the weighted Hilbert space 
$	L^{2,\alpha}_\beta(\mathbb{D})=L^{2}(\mathbb{D},d\mu_{\alpha,\beta})
$ 
 of all complex-valued functions that are square integrable  with respect to the positive measure \begin{align}\label{measure}
 	d\mu_{\alpha,\beta}(z):=(1-|z|^2)^{\alpha} |z|^{2\beta}dxdy; \quad    z=x+iy, \, \, \alpha,\beta >-1.\end{align}
The main results concerning the functions $\psi^{\gamma,\eta}_{m,n} $ are summarized in Theorem \ref{Mth3}. Namely, we prove that they form an orthogonal system of eigenfunctions in $L^{2,\alpha}_\beta(\mathbb{D})$  for a perturbed magnetic Laplacian, which is essentially the classical magnetic Schr\"odinger operator on the hyperbolic disc perturbed by a  {particular potential (with zero magnetic field) modeling the Aharonov--Bohm effect (see Remark \ref{rempotential})}. 
Moreover, the $L^{2,\alpha}_\beta$-eigenspace of the considered Laplacian  
associated with its lowest eigenvalue is shown to be the $\beta$-modified Bergman space $\mathcal{A}^{2,\alpha}_{\beta}(\mathbb{D})$ on the punctured unit disc $\mathbb{D}^* $ recently introduced and studied in \cite{GhanmiSnoun2021,GhiloufiSnoun2020
}.
The other $L^{2,\alpha}_\beta$-eigenspaces associated with the hyperbolic Landau levels for the considered Laplacian can be seen as the polyanalytic analogs of $\mathcal{A}^{2,\alpha}_{\beta}(\mathbb{D})$ 
(see Remark \ref{rembetaBerg}).

The motivation of considering $\psi^{\gamma,\eta}_{m,n}$ is that they can be seen as the spectral side of fractional Zernike functions. For  special values of $\gamma$ and $\eta$ they  are closely connected to by
\begin{align} \label{Zernike1}
	\mathcal{Z}^{\kappa_m,\rho}_{m,n}(z,\bz)	&= 	 |z|^{2\eta} (1-|z|^2)^{\frac{\alpha+1-\kappa_m}{2}}  	\psi^{\gamma,\eta}_{m,n}(z,\bz) 
	\end{align}
for $m,n \geq 0$ with $\rho = \beta-2\eta$ and for $\kappa$ depending in $m$ and given by $\kappa=\kappa_m = \alpha-2(\gamma+m)-1$.
However, this last fact can not be employed to recover the global properties of the fractional Zernike functions   $\mathcal{Z}^{\kappa,\rho}_{m,n}(z,\bz)$. Only the local ones for every fixed $m$, $n$ and $\rho$ with the specific $\kappa=\kappa_m$ can be derived.

 For the concrete study of $\mathcal{Z}^{\kappa,\rho}_{m,n}(z,\bz)$ we begin by establishing 
  their explicit expressions, their different hypergeometric representations, their expression in terms of the Jacobi polynomials as well as their connection to the complex Zernike polynomials in \eqref{DiscePol}.
 Subsequently, the zero sets of $\mathcal{Z}^{\kappa,\rho}_{m,n}(z,\bz)$ are described (Corollary \ref{corzeros}) and 
 shown to be the centered circles of radii being the zeros of the real Jacobi polynomials.
The orthogonality in the Hilbert space $L^{2,\kappa}_{\rho}(\D)= L^{2}(\mathbb{D},d\mu_{\kappa,\rho}$) is discussed and the square norm is explicitly computed. 
The membership to a specific class of poly-meromorphic functions in $\D$ is also considered (Theorem \ref{thmAnaMeremor}). 
 Moreover,  we investigate the operational formulas they satisfy including those of Burchnall type and discuss some recurrence relations, the differential equations they obey (Theorems \ref{thmDiffeqs} and \ref{thmDiffeq}) and so on. 
Certain associated generating functions are obtained such a bilinear generating function analogous to the one  Hardy--Hille generating function for the generalized Laguerre polynomials.
The latter one can be employed to derive special 
 integral representation for $\mathcal{Z}^{\kappa,\rho}_{m,n}(z,\bz)$. Another  integral representation of Cauchy-type is obtained as a special integral on the unit circle.  
%
Finally, we show in Theorem \ref{Mth5}  that the truncated fractional Zernike functions 
\begin{align} \label{trabbasis} \varUpsilon_{m,s}^{\kappa,\rho} (z,\bz) := z^s|z|^{-s}\mathcal{Z}^{\kappa,\rho}_{m,m}(z,\bz), \, s\in \Z,  \,  m=0,1,\cdots,
\end{align}
 constitute an orthogonal basis of the Hilbert space $L^{2,\kappa}_{\rho}(\D)$. 
Accordingly, we define a second class of poly-meromorphic Bergman spaces leading to a complete microlocal orthogonal decomposition of the underlying Hilbert space $L^{2,\kappa}_\rho(\D)$.  
The obtained results will contribute efficiently in the study of the associated isometric integral transforms (of Bargmann type) on the configuration space on the positive half real line. 

The remaining sections are organized as follows. Section 2 deals with the spectral realization of the $\beta$-modified functions $	\psi^{\gamma,\eta}_{m,n}$ by Schr\"odinger's factorization method.   
A 
proof of $\psi^{\gamma,\eta}_{m,n}$ form an orthogonal system of eigenfuctions in $L^{2,\alpha}_\beta(\mathbb{D})$ is also presented in this section.
 The basic properties of the fractional functions as described above are  stated and proved in Section 3. 


\section{The $\beta$-restricted Zernike functions (spectral realization)}

In this section we are concerned with the functions 
\begin{align}\label{Zernike0}
	\psi^{\gamma,\eta}_{m,n}(z,\bz)
	&=	(-1)^m z^{\eta-\beta} \bz^{-\eta} (1-|z|^2)^{\gamma-\alpha+m}
	\frac{\partial^m}{\partial z^m}  \big( z^{n+\beta-2\eta} (1-|z|^2)^{\alpha-2\gamma-m-1}\big)
\end{align}
for given reals $\alpha,\beta,\gamma,\eta$. They  
referred to as the $\beta$-restricted Zernike functions (justified by Remark \ref{rembetaBerg} below).
 We aim to derive their basic properties and show that they form an  orthogonal system  
of $L^{2,\alpha}_{\beta}$-eigenfunctions for a perturbed magnetic Laplacian of the form
\begin{align}\label{explicitLG}
	\Delta_{a,b}^{c,d}&=
	\Delta_{hyp}	
	+   (1-|z|^2)\left( H_{a}^b(z) E	
	- 	H_{c}^d(z) \overline{E} \right)   
	+  H_{a}^b(z)H_{c}^d(z) |z|^2
\end{align}
 acting on the weighted Hilbert space $L^{2,\alpha}_\beta(\mathbb{D})$, $\alpha,\beta >-1$.
%
Above  $a$, $b$, $c$ and $d$ are given real numbers,  $\Delta_{hyp} = -   (1-|z|^2)^2{\partial ^2}/{\partial z		\partial {\bar z}} $
is the Laplace--Beltrami operator on the hyperbolic disc,  $\overline{E}= \bz {\partial}/{\partial \bz}$ denotes the complex conjugate of the complex Euler operator $E  := z{\partial}/{\partial z} 
$ and  
\begin{align}\label{Hab} 
	H_{a}^{b}(z) :=  a + b -   \frac{b}{|z|^2} .
\end{align} 
 It is worth noting that for particular values of $a,b,c,d$ one recovers the magnetic Schr\"odinger operator on the hyperbolic unit disc representing the Hamiltonian of a charged particle in motion under an external uniform magnetic field \cite{Comtet1987,BoussjraIntissar1998,GhInJmp2005,Graham1983}. 

To this end,  we have to factorize the considered Laplacian in terms of some first order differential operators (leading in particular to their Rodrigues type formula). Thus, if we set $h_{\alpha,\beta}(z) = h(z)^{\alpha} |z|^{2\beta}$ with $h(z)= 1-|z|^2 $, we can consider the first order differential operator
$$ A_{\gamma,\eta}f(z):=h_{1-\gamma,-\eta}(z)\frac{\partial}{\partial \bz}(h_{\gamma,\eta} f)(z)$$
for given fixed reals $\gamma$ and $\eta$.
Its explicit expression is given by   
\begin{equation}\label{Diffop}
	A_{\gamma,\eta}f(z)= \left\{ (1-|z|^2)  \frac{\partial}{\partial \bz}
	-
	H_{\gamma}^{\eta}(z)\right\} f(z).
\end{equation}
The corresponding null space is closely connected to the set $\text{Hol}(\mathbb{D}^*)$ of holomorphic functions on the punctured unit disc. Namely, we have $\ker(A_{\gamma,\eta})= h_{-\gamma,-\eta}\text{Hol}(\mathbb{D}^*)$.	
Moreover, the formal adjoint operator $A^{*_{\alpha,\beta}}_{\gamma,\eta}$ of $A_{\gamma,\eta}$ with respect to the inner scalar product 
\begin{align}\label{spbeta} \scal{f,g}_{\alpha,\beta} :=  \int_{\mathbb{D}} f(z) \overline{g(z)} d\mu_{\alpha,\beta}(z)
\end{align}
in $L^{2,\alpha}_\beta(\mathbb{D})$ is given by 
\begin{align} 
	A^{*_{\alpha,\beta}}_{\gamma,\eta}f(z) &:= -h_{\gamma-\alpha,\eta-\beta}(z) \frac{\partial}{\partial z} (h_{\alpha-\gamma+1,\beta-\eta} f)(z) 
\end{align}
in account of the conventional calculation.
Accordingly, we perform  
\begin{align}\label{LpLm}
	\mathcal{L}_{\gamma,\eta}^{\alpha,\beta,+}= A_{\gamma,\eta}A^{*_{\alpha,\beta}}_{\gamma,\eta}\quad and \quad \mathcal{L}_{\gamma,\eta}^{\alpha,\beta,-}=A^{*_{\alpha,\beta}}_{\gamma,\eta}A_{\gamma,\eta}.
\end{align}
Straightforward computation leads to the explicit expression of these second order differential operators in terms of $\Delta_{a,b}^{c,d}$  in \eqref{explicitLG}
(we omit the proof).

\begin{lemma}\label{th1}
	The expression of  $\mathcal{L}_{\gamma,\eta} ^{\alpha,\beta,+}$ in the $z$-coordinate is given by  	
	\begin{align*}
		\mathcal{L}_{\gamma,\eta}^{\alpha,\beta,+} &= 
		\Delta_{\gamma+1,\eta}^{\gamma-\alpha-1,\eta-\beta}  + (\alpha-\gamma+1) .
		\nonumber
	\end{align*}
	Moreover, the operators $\mathcal{L}_{\gamma,\eta}^{\alpha,\beta,+}$ and $\mathcal{L}_{\gamma,\eta}^{\alpha,\beta,-}$ satisfy 
	\begin{align} 
		\label{link}
		\mathcal{L}_{\gamma,\eta}^{\alpha,\beta,+}=\mathcal{L}_{\gamma+1,\eta}^{\alpha,\beta,-} + (\alpha-2\gamma)     . 
	\end{align} 
\end{lemma}

\begin{remark} 
	For $\alpha=-2$ and $\beta =0$ we have $\mathcal{L}_{\gamma,\eta}^{\alpha,\beta,+}=\mathcal{L}_{\gamma+1,\eta}^{\alpha,\beta,-}  -2(\gamma+1) $. Also  $H_{\alpha-\gamma+1}^{\beta-\eta} = -H_{\gamma+1}^{\eta}$ so that the Laplacian $\mathcal{L}_{\gamma,\eta}^{\alpha,\beta,+}$ reduces further to 
	\begin{align}\label{Lap+red3}
		\mathcal{L}_{\gamma,\eta}^{\alpha,\beta,+} 
		&= - h \left\{h \frac{\partial^2}{\partial z \partial \bz} -  H_{\gamma+1}^{\eta}(z)\left(   E  - \overline{E}\right)  \right\}  +  (H_{\gamma+1}^{\eta}(z))^2  |z|^2   -(\gamma+1 ).
	\end{align}
	For the particular cases of $\gamma,\eta$ we recover the Landau-like Hamiltonian on $\D$ (see e.g. \cite{BoussjraIntissar1998,GhInJmp2005,Graham1983}). 
\end{remark}

\begin{remark} \label{rempotential}
	The considered operators $\mathcal{L}_{\gamma,\eta}^{\alpha,\beta,+}$ and $\mathcal{L}_{\gamma,\eta}^{\alpha,\beta,-}$ can be realized geometrically as  magnetic Schr\"odinger operators associated with a singular real differential $1$-form (vector potential) $\theta_{\alpha,\beta}= \theta_{\alpha} + \widetilde{\theta_{\beta}}$ with $d \widetilde{\theta_{\beta}}=0$ and $d\theta_{\alpha}$ is the Kh\"aler two form on the hyperbolic unit disc up to a multiplicative constant. More precisely, we have
	\begin{align}  
		\theta_{\alpha,\beta}(z)&=  \frac{i \alpha \left( {\bar
				{z}}dz - zd\bar z \right) }{1
			-|z|^2}	 - i\beta \left( \frac{dz}{z}  - \frac{d\bz}{\bz} \right) .
	\end{align}  
\end{remark}

Now, by means of the identity \eqref{link} we can establish the following (we omit the proof). 

\begin{lemma}\label{CorIdComm}
	The following commutation rules hold trues 
	\begin{align*} 
		(i)	& \qquad \mathcal{L}_{\gamma,\eta}^{\alpha,\beta,+} A_{\gamma,\eta} =  A_{\gamma,\eta} \mathcal{L}_{\gamma,\eta}^{\alpha,\beta,-} \quad \mbox{and} \quad  A^{*_{\alpha,\beta}}_{\gamma,\eta}\mathcal{L}_{\gamma,\eta}^{\alpha,\beta,+} =   \mathcal{L}_{\gamma,\eta}^{\alpha,\beta,-} A^{*_{\alpha,\beta}}_{\gamma,\eta} .
		\\
		(ii)   & \qquad	\mathcal{L}_{\gamma,\eta}^{\alpha,\beta,+} A^{*_{\alpha,\beta}}_{\gamma+1,\eta}=A^{*_{\alpha,\beta}}_{\gamma+1,\eta}\left( \mathcal{L}_{\gamma+1,\eta}^{\alpha,\beta,+}  + (\alpha-2\gamma)  \right)  . \\
		(iii)   & \qquad	A_{\gamma+1,\eta}\mathcal{L}_{\gamma,\eta}^{\alpha,\beta,+} =\left( \mathcal{L}_{\gamma+1,\eta}^{\alpha,\beta,+}  + (\alpha-2\gamma) \right)  A_{\gamma+1,\eta} .\\
		(iv)   &  \qquad 	\mathcal{L}_{\gamma+1,\eta}^{\alpha,\beta,-}A_{\gamma,\eta} 
		= A_{\gamma,\eta}\left( \mathcal{L}_{\gamma,\eta}^{\alpha,\beta,-}   - (\alpha-2\gamma)   \right) . \\
		(v)   &  \qquad 	A^{*_{\alpha,\beta}}_{\gamma,\eta} \mathcal{L}_{\gamma+1,\eta}^{\alpha,\beta,-} 
		= \left(\mathcal{L}_{\gamma,\eta}^{\alpha,\beta,-}   - (\alpha-2\gamma) \right) A^{*_{\alpha,\beta}}_{\gamma,\eta}    . 
	\end{align*}
\end{lemma}


Lemma \ref{CorIdComm} is efficient in analyzing and studying the family of functions $	\psi^{\gamma,\eta}_{m,n}$ in \eqref{Zernike0}.
In fact, we show that they can be obtained by successive application of $A^*_{\gamma+j,\eta}$; $j= 1,2,\cdots ,m$, to the ground state functions. Namely, we consider the differential operator   
\begin{align*}
	A^{*,m}_{\gamma,\eta}(f) & := A^{*_{\alpha,\beta}}_{\gamma+1,\eta}\circ A^{*_{\alpha,\beta}}_{\gamma+2,\eta}\circ \cdots \circ A^{*_{\alpha,\beta}}_{\gamma+m,\eta}(f).
\end{align*}

Then, we claim the following. 

\begin{lemma}\label{lemdRodpsi}
	The closed expression of Rodrigues type for $	A^{*,m}_{\gamma,\eta}$ is given by 
	\begin{align}\label{hdm}
		A^{*,m}_{\gamma,\eta}(f) 
		&=(-1)^m h_{\gamma-\alpha+m,\eta-\beta}  \frac{\partial^m}{\partial z^m}  (h_{\alpha-\gamma ,\beta-\eta} f) .
	\end{align}
\end{lemma}

\begin{proof}
	Starting from the definition of $A^{*,m}_{\gamma,\eta}$ one gets
	\begin{align*}
		A^{*,m}_{\gamma,\eta}(f)
		&=(-1)^m h_{\gamma-\alpha-1,\eta-\beta}\left( h^2 \frac{\partial}{\partial z}\right) ^m (h_{\alpha-\gamma-m+1,\beta-\eta} f)  .
	\end{align*}
	Then \eqref{hdm} readily follows thanks to the fact 
	$(h^2 \partial)^m (f) =h^{m+1}
	\partial^m(h^{m-1}f)$  in \cite{Ghanmi2010}.
\end{proof}


The main result in this section is the following.

\begin{theorem}\label{Mth3} Fix $\gamma$ such that $\alpha > 2\gamma+1$. Then, for integers $m,n$ such that $n>2\eta-\beta-1$ and $ 0\leq m < (\alpha-1 -2\gamma)/2$, the following assertions hold.
	\begin{itemize}
		\item[(i)] The function $\psi^{\gamma,\eta}_{m,n}$ is a  $L^{2,\alpha}_{\beta}$-eigenfunction of $\mathcal{L}_{\gamma,\eta}^{\alpha,\beta,+}$ with $ E^{\gamma,\alpha}_m = 	(m+1)(\alpha-2\gamma-m)$ as corresponding eigenvalue.
		\item[(ii)] The functions $\psi^{\gamma,\eta}_{m,n}$
		form an orthogonal system in the Hilbert space $L^{2,\alpha}_{\beta}(\mathbb{D})$ and their square norm (induced from \eqref{spbeta}) is given by
		\begin{align}\label{normZernike}
			\norm{\psi^{\gamma,\eta}_{m,n}}_{\alpha,\beta}^2 = 
			\frac{\pi m!}{(\alpha-2(\gamma+m)-1)} 	 
			\frac{ \Gamma(\alpha-2\gamma -m)\Gamma(n+\beta-2\eta+1) }{\Gamma(n + \alpha +\beta -2(\gamma+\eta+m))}.
		\end{align}
	Here $\Gamma$ is the Gamma Euler function.
	\end{itemize}
\end{theorem} 

\begin{proof}
	In virtue of the algebraic identity $(iii)$ in Lemma \ref{CorIdComm} and the identity \eqref{link} we can proceed by mathematical induction to get
	\begin{align*}
		\mathcal{L}_{\gamma,\eta}^{\alpha,\beta,+} A^{*,m}_{\gamma,\eta}
		&= A^{*,m}_{\gamma,\eta} \mathcal{L}_{\gamma+m,\eta}^{\alpha,\beta,+} +\sum_{j=0}^{m-1}(\alpha-2(\gamma+j)) A^{*,m}_{\gamma,\eta} \\
		&= A^{*,m}_{\gamma,\eta} \left( \mathcal{L}_{\gamma+m+1,\eta}^{\alpha,\beta,-} + (\alpha-2(\gamma+m))\right)  +\sum_{j=0}^{m-1}(\alpha-2(\gamma+j)) A^{*,m}_{\gamma,\eta} \\
		&= A^{*,m}_{\gamma,\eta} \mathcal{L}_{\gamma+m+1,\eta}^{\alpha,\beta,-}  +\sum_{j=0}^{m}(\alpha-2(\gamma+j)) A^{*,m}_{\gamma,\eta} \\&=
		A^{*,m}_{\gamma,\eta} \mathcal{L}_{\gamma+m+1,\eta}^ {\alpha,\beta,-} + (m+1)(\alpha-2\gamma-m) A^{*,m}_{\gamma,\eta} .
	\end{align*}
	Accordingly, it becomes clear that the functions  $A^{*,m}_{\gamma,\eta}(\varphi^{\gamma,\eta}_{m})$ are eigenfunctions of $\mathcal{L}_{\gamma,\eta}^{\alpha,\beta,+}$ whenever $\varphi^{\gamma,\eta}_{m}$  belongs to the null space of
	$A_{\gamma+m+1,\eta}$, 
	$$\ker(A_{\gamma+m+1,\eta})=\{f: \mathbb{D}^* \longrightarrow \C; \, \, A_{\gamma+m+1,\eta} f=0\} \subseteq \ker \left( \mathcal{L}_{\gamma+m+1,\eta}^{-}\right).$$
	This is the case when considering 
	\begin{eqnarray}\label{eq1}
		\varphi^{\gamma,\eta}_{m}(z) = \varphi^{\gamma,\eta}_{m,n}(z) : = z^n (1-|z|^2)^{-(\gamma+m+1)}|z|^{-2\eta}; \ \  n\in \Z  .
	\end{eqnarray} 
	More precisely, the functions   $A^{*,m}_{\gamma,\eta}(\varphi^{\gamma,\eta}_{m})=A^{*,m}_{\gamma,\eta}(z^n h_{-(\gamma+m+1),-\eta})$ 
	are given by			\begin{align}\label{Zernike01}
		A^{*,m}_{\gamma,\eta}&(z^n h_{-(\gamma+m+1),-\eta})
		= (-1)^m h_{\gamma-\alpha+m,\eta-\beta}  	\frac{\partial^m}{\partial z^m}  (z^n  h_{\alpha-2\gamma-m-1 ,\beta-2\eta} )
		\\&=
		(-1)^m (1-|z|^2)^{\gamma-\alpha+m}|z|^{	2(\eta-\beta)} \frac{\partial^m}{\partial z^m}  \big(z^n |z|^{2(\beta-2\eta)} (1-|z|^2)^{\alpha-2(\gamma+m)+m-1}\big) \nonumber
	\end{align}
	thanks to Lemma \ref{lemdRodpsi}. 
	The latter formula reduces further to the expression of the $\beta$-restricted Zernike functions  in \eqref{Zernike0}. Moreover,  they
	satisfy 
	\begin{align}\label{eigenfunct} \mathcal{L}_{\gamma,\eta}^{\alpha,\beta,+} \left( \psi^{\gamma,\eta}_{m,n}\right)=(m+1)(\alpha-2\gamma-m)\psi^{\gamma,n}_{m,n} = E^{\gamma,\alpha}_m \psi^{\gamma,\eta}_{m,n}.
	\end{align}
	Now, for their orthogonality in $L^{2,\alpha}_{\beta}(\mathbb{D})$ one can use their explicit expressions in terms of certain special functions (see for example Remark \ref{rembetaZJacobi} below). However, we present below another proof using the factorization method. To this purpose, notice first that $A^{*,m}_{\gamma,\eta} = A^{*_{\alpha,\beta}}_{\gamma+1,\eta}\circ A^{*,m-1}_{\gamma+1,\eta}$ and that $\varphi^{\gamma,\eta}_{m,n}= \varphi^{\gamma+1,\eta}_{m-1,n}$. It follows
	$$\psi^{\gamma,\eta}_{m,n}= 
	A^{*,m}_{\gamma,\eta}(\varphi^{\gamma,\eta}_{m,n})
	=	A^{*_{\alpha,\beta}}_{\gamma+1,\eta}\circ A^{*,m-1}_{\gamma+1,\eta}(\varphi^{\gamma,\eta}_{m,n})
	= A^{*_{\alpha,\beta}}_{\gamma+1,\eta} (\psi^{\gamma+1,\eta}_{m-1,n}) .$$
	Accordingly, making use of \eqref{eigenfunct} we obtain   
	\begin{align*}
		\scal{\psi^{\gamma,\eta}_{m,n},\psi^{\gamma,\eta}_{j,k}} 
		&	= \scal{ \mathcal{L}_{\gamma+1,\eta}^{\alpha,\beta,+}  ( \psi^{\gamma+1,\eta}_{m-1,n}  ) , \psi^{\gamma+1,\eta}_{j-1,k}} 
		= E^{\gamma+1,\alpha}_{m-1} \scal{ \psi^{\gamma+1,\eta}_{m-1,n},  \psi^{\gamma+1,\eta}_{j-1,k}} .
	\end{align*} 
	More generally, by induction we arrive at 
	\begin{align*}
		\scal{\psi^{\gamma,\eta}_{m,n},\psi^{\gamma,\eta}_{j,k}} 
		&= \prod_{\ell=1}^{s} 
		E^{\gamma+\ell,\alpha}_{m-\ell}  \scal{ \psi^{\gamma+s,\eta}_{m-s,n},  \psi^{\gamma+s,\eta}_{j-s,k}} ; \, 1 \leq s \leq m.
	\end{align*}
	Therefore, without lost of generality we can assume that $m\leq j$ and take $s=m$ to get 
	\begin{align*}
		\scal{\psi^{\gamma,\eta}_{m,n},\psi^{\gamma,\eta}_{j,k}} 
		&= \prod_{\ell=1}^{m} 
		E^{\gamma+\ell,\alpha}_{m-\ell}  \scal{ \psi^{\gamma+m,\eta}_{0,n},  \psi^{\gamma+m,\eta}_{j-m,k}} \\ 
		&= \prod_{\ell=1}^{m} 
		E^{\gamma+\ell,\alpha}_{m-\ell}\scal{ \psi^{\gamma+m,\eta}_{0,n}, A^*_{\gamma+m+1,\eta}\circ A^*_{\gamma+m+2,\eta}\circ \cdots \circ A^*_{\gamma+j,\eta}( \varphi^{\gamma+m,\eta}_{j-m,k})} \\
		&= \prod_{\ell=1}^{m} 
		E^{\gamma+\ell,\alpha}_{m-\ell}\scal{  A_{\gamma+j,\eta}\circ \cdots \circ  A_{\gamma+m+1,\eta} 	(\varphi^{\gamma+m,\eta}_{0,n}),  \varphi^{\gamma+m,\eta}_{j-m,k} } .
	\end{align*}
	The last identity holds by observing that 
	$ \psi^{\gamma+s,\eta}_{0,n} = 	\varphi^{\gamma+s,\eta}_{0,n}$, which readily follows  from \eqref{hdm} or \eqref{Zernike01}. 
	Next, since $\varphi^{\gamma+m,\eta}_{0,n}
	$ belongs to $\ker (A_{\gamma+m+1,\eta})$ and then $ A_{\gamma+j,\eta}\circ \cdots \circ  A_{\gamma+m+2,\eta}\circ A_{\gamma+m+1,\eta}
	(\varphi^{\gamma+m,\eta}_{0,n})$ vanishes whenever $m<j$, we obtain 
	\begin{align*}
		\scal{\psi^{\gamma,\eta}_{m,n},\psi^{\gamma,\eta}_{j,k}} 
		&= \left( \prod_{\ell=1}^{m} 		E^{\gamma+\ell,\alpha}_{m-\ell}  \right) 
		\scal{  \varphi^{\gamma+m,\eta}_{0,n}, \varphi^{\gamma+m,\eta}_{0,k}} \delta_{m,j}.
	\end{align*}
	To the computation of the  quantity $\scal{  \varphi^{\gamma+m,\eta}_{0,n}, \varphi^{\gamma+m,\eta}_{0,k}}$ we make use of \eqref{eq1} giving the explicit expression of $\varphi^{\gamma,\eta}_{m,n}$.   
	This yields
	\begin{align*}
		\scal{  \varphi^{\gamma+m,\eta}_{0,n}, \varphi^{\gamma+m,\eta}_{0,k}} &= 
		\int_{\mathbb{D}} (1-|z|^2)^{\alpha-2(\gamma+m+1)} |z|^{2(\beta-2 \eta)}z^n \bz^k d\lambda(z) \\
		&=\pi\left( \int_{0}^1 (1-t)^{\alpha-2(\gamma+m+1)} t^{n+\beta-2 \eta}   dt\right) \delta_{n,k} 
		\\ &=\pi B(n+\beta-2 \eta+1, \alpha-2(\gamma+m)-1)
		\delta_{n,k},
	\end{align*}
	where $B(a,b)$ denotes the classical beta function. 
	The validity of the previous formula requires that $n>2 \eta-\beta-1$ and $\alpha-2\gamma-1>2m$ with $\alpha-2\gamma-1>0$. 
	Finally, since 
 $$ \prod_{\ell=1}^{m} 		E^{\gamma+\ell,\alpha}_{m-\ell}   = m!(\alpha-2(\gamma +m))_m = m! \frac{\Gamma(\alpha-2\gamma -m)}{\Gamma(\alpha-2(\gamma +m))} , $$ we arrive at   
	$$\scal{\psi^{\gamma,\eta}_{m,n},\psi^{\gamma,\eta}_{j,k}}
	=   \frac{\pi m!}{(\alpha-2(\gamma+m)-1)} 	 
	\frac{ \Gamma(\alpha-2\gamma -m)\Gamma(n+\beta-2\eta+1) }{\Gamma(n + \alpha +\beta -2(\gamma+\eta+m))}
	\delta_{m,j}\delta_{n,k}.$$ 
	This completes the proof.
\end{proof}

\begin{remark}\label{rembetaBerg}
The functions in \eqref{Zernike0} corresponding to $\gamma=-1$, $\eta=0$ and $m=0$ reduce further to  $ \psi^{-1,0}_{0,n}(z,\bz) =  z^n $, for varying integer $n > - (\beta+1)$, whose  square norm in $L^{2,\alpha}_\beta(\mathbb{D})$ is given by 
	$$ \norm{\psi^{-1,0}_{0,n}}_{\alpha,\beta}^2 
	=\pi 	 
	\frac{ \Gamma(\alpha+1 )\Gamma(n+\beta+1) }{\Gamma(n + \alpha +\beta +2)} .$$
	 They form an orthogonal basis of the $\beta$-modified Bergman space $\mathcal{A}^{2,\alpha}_{\beta}(\D)$ defined as the closed subspace in $L^{2,\alpha}_\beta(\mathbb{D})$ formed by the holomorphic functions on the punctured disc $\mathbb{D}^{*}$ (see \cite{GhanmiSnoun2021,GhiloufiSnoun2020
	} for details). 
	In other words, the $\beta$-modified Bergman space is the $L^2$-eigenspace of our magnetic Laplacian $ \mathcal{L}_{\gamma+1,\eta}^{\alpha,\beta,+}$ associated with its lowest Landau level.
	For the particular case of $\beta=0$ we recover the classical Bergman space on the unit disc with respect to the weight function being of the generalized Gegenbauer form $(1-|z|^2)^{\alpha}$. 
\end{remark}

\begin{remark}\label{remNoyOrth}
	The functions $\psi^{\gamma,\eta}_{m,n}$ do not form a complete system  in $L^{2,\alpha}_\beta(\mathbb{D})$.
	However, for fixed $m$ such that  $ 0\leq m < (\alpha-1 -2\gamma)/2$ and varying integer  $ n \geq  2\eta-\beta$ they span a specific closed subspace $\mathcal{A}^{2,\alpha}_{\beta,m}(\mathbb{D}) $ in $L^{2,\alpha}_{\beta}(\mathbb{D})$. This gives rise to what can be called the $m$-th generalized (or also poly-meromophic)  $\beta$-modified  Bergman space on $\mathbb{D}^{*}$ and can be seen  as  the polyanalytic analog of the $\beta$-modified Bergman space. Its reproducing kernel is given in Remark \ref{remRepKer} below.
\end{remark}

%
%
%
%


\section{Fractional Zernike functions}

In this section we provide an accurate  theoretical study for the fractional Zernike functions in \eqref{betZpol}.
 We discuss their connection to some special functions, zeros, orthogonality in $L^{2,\kappa}_{\rho}(\D)$, regularity, differential equations, recurrence and operational formulas. 
Some results concerning the generating functions, the integral representations and completeness are also obtained.

\subsection{Connection to special functions and explicit expression.}

We begin by establishing the explicit expression of the fractional Zernike functions $\mathcal{Z}^{\kappa,\rho}_{m,n}$ in terms of the classical Zernike polynomials.  
Thus, for given real $b$ and  nonnegative integer $m$ we define the infected minimum $m\wedge^* b$ to be
$$ m\wedge^* b = \left\{\begin{array}{ll}
	\min(m, b), &   \,  b=0,1,2, \cdots \\
	m, &  b\in \R , \, b\ne 0,1,2, \cdots  .\\
\end{array} \right.$$

\begin{proposition}\label{propZZ}
	For every $\rho >-1$ we have 
	\begin{align}\label{Zz} \mathcal{Z}^{\kappa,\rho}_{m,n}(z,\bz)&= \frac{ m!\Gamma(\rho+1)}{(\kappa+m+1)_n }  \sum_{j=0}^{m\wedge^* \rho}  \frac{ (-1)^{j} }{j!(m-j)!\Gamma(\rho-j+1)} 
		\frac{\left( 1-|z|^2  \right)^j}{z^{j}}   \mathcal{Z}^{\kappa+j}_{m-j,n}(z,\bz)   . 
	\end{align} 
\end{proposition}

\begin{proof}
	Using the facts	\eqref{derjmonomial}  and	 
	$$ z^n (1-|z|^2)^{\kappa+m} = \frac{(-1)^n}{(\kappa+m+1)_n} \frac{\partial^n}{\partial \bz^n} \left(  (1-|z|^2)^{\kappa+m+n}\right) 
	$$
	we get
	\begin{align*} \mathcal{Z}^{\kappa,\rho}_{m,n}(z,\bz)
		&= \frac{(-1)^{m+n}}{(\kappa+m+1)_n} z^{-\rho} (1-|z|^2)^{-\kappa}  \frac{\partial^m}{\partial z^m}\left(  z^\rho \frac{\partial^n}{\partial \bz^n} \left(  (1-|z|^2)^{\kappa+m+n}\right)\right)\\
		&= \frac{(-1)^{m+n}m!}{(\kappa+m+1)_n} z^{-\rho} (1-|z|^2)^{-\kappa}  \sum_{j=0}^m \frac{(-\rho)_{j}}{j!(m-j)!}    z^{\rho-j} \frac{\partial^{m-j+n}}{\partial z^{m-j} \partial \bz^n} \left(  (1-|z|^2)^{\kappa+j+m-j+n}\right),
	\end{align*}
	which can be rewritten as \eqref{Zz}.
\end{proof}

\begin{remark}\label{remZZ}
	For $\rho=0$ we recover the Zernike polynomials $\mathcal{Z}_{m,n}^\kappa (z,\bz)$ up to the multiplicative constant 
	$1/(\kappa+m+1)_{n}$, 	
	while when $\rho=1$ we get
	\begin{align*}   
		\mathcal{Z}_{m,n+1}^\kappa (z,\bz)&= (\kappa+m+n+1) \left( z         \mathcal{Z}^{\kappa}_{m,n}(z,\bz)	
		+m   \left( 1-|z|^2  \right)    \mathcal{Z}^{\kappa+1}_{m-1,n}(z,\bz)\right),
	\end{align*}
	which is exactly the three terms recurrence formula for the Zernike polynomials \cite[p. 403, Eq. (5.1)]{Aharmim2015}. This follows since that  $(-\rho)_j = 0$ for $j\geq \rho+1$ whenever $\rho = 0,1,2, \cdots$.
	More generally, from \eqref{frZcZ} with $\rho$ being a nonnegative integer we obtain new recurrence formula 
	 for the classical complex Zernike polynomials	
	\begin{align}  \mathcal{Z}_{m,n+\rho}^\kappa (z,\bz) 
		&	=  m!\Gamma(\rho+1) (\kappa+m+n+1)_\rho   \sum_{j=0}^{m\wedge \rho}  \frac{ (-1)^{j} z^{\rho-j}\left( 1-|z|^2  \right)^j}{j!(m-j)!\Gamma(\rho-j+1)}    \mathcal{Z}^{\kappa+j}_{m-j,n}(z,\bz). 
	\end{align}
\end{remark}

The explicit expression of the few first terms of $\mathcal{Z}^{\kappa,\rho}_{m,n}(z,\bz)$ can be computed easily from the Rodrigues formula \eqref{betZpol} or also using \eqref{propZZ}. Thus, those corresponding to $m=0$ reduce to the monomials, $\mathcal{Z}^{\kappa,\rho}_{0,n}(z,\bz)= z^n $.
For $m=1$ and $m=2$ we get respectively  
\begin{align*}
	\mathcal{Z}^{\kappa,\rho}_{1,n}(z,\bz)
	&	=   (\kappa+n_\rho+1) \bz z^n    - n_\rho z^{n-1}   
\end{align*}
and 
\begin{align*}
	\mathcal{Z}^{\kappa,\rho}_{2,n}(z,\bz)
	&= (\kappa+n_\rho+1)(\kappa+n_\rho+2)\bz^{2}z^{n}  - 2 n_\rho  (\kappa+n_\rho+1)\bz z^{n-1}   +   n_\rho(n_\rho-1)z^{n-2}    ,
\end{align*}
where we have set $n_\rho = n+\rho$. 
A general formula for the explicit expression of $\mathcal{Z}^{\kappa,\rho}_{m,n}(z,\bz)$ is given by the following assertion.

\begin{proposition}\label{propZexp}	
	We have
	\begin{align}\label{Zmnexp} \mathcal{Z}^{\kappa,\rho}_{m,n}(z,\bz)&= \sum_{j=0}^{m\wedge^*(n+\rho)}   \frac{(-1)^{j}  m! \Gamma(n+\rho+1)\Gamma(\kappa+m+1)  }{j! (m-j)!\Gamma(n+\rho-j+1) \Gamma(\kappa+j+1) } \left(  1-|z|^2\right)^{j} z^{n-j }  \bz^{m-j} .  
	\end{align}
\end{proposition}

\begin{proof}
	Using the fact  $(x)^{\underline{n}} = \Gamma(x+1)/\Gamma(x-n+1)$ for the decreasing factorial $(x)^{\underline{n}} = x(x-1)\cdots (x-n+1)$, we obtain
	$$ \frac{\partial^m}{\partial z^m}  \left(  1-xz\right)^{a}  
	= (-1)^m \frac{\Gamma(a+1)}{\Gamma(a+1-m)}x^m \left(  1-xz\right)^{a-m} $$
	and 
	\begin{align}\label{derjmonomial}
		\frac{\partial^m}{\partial z^m}  \left( z^{a} \right) 
		= (-a)_m  z^{a-m} = \varepsilon^*_{a,m}  \frac{\Gamma(a+1)}{\Gamma(a+1-m)} z^{a-m},
	\end{align}
	where for the nonnegative integer $m$ we have set
	$$\varepsilon^*_{a,m} = \left\{\begin{array}{lll}
		1, & a\geq m; \,  a=0,1, \cdots \\
		0, & a< m; \,  a=0,1, \cdots\\
		1, &  a\in \R, \, a \ne 0,1, \cdots. \\
	\end{array} \right.$$
	Thus, applying the Leibnitz formula for high order derivation of a product yields
	\begin{align}
		\mathcal{Z}^{\kappa,\rho}_{m,n}(z,\bz) &=  		(-1)^m z^{-\rho} (1-|z|^2)^{-\kappa} \frac{\partial^m}{\partial z^m}  \big( z^{n+\rho} (1-|z|^2)^{\kappa+m}\big)\nonumber\\
		&=    \sum_{j=0}^m \varepsilon^*_{n+\rho,j} \frac{(-1)^{j} m!\Gamma(n+\rho+1)\Gamma(\kappa+m+1)  }{j!(m-j)!  \Gamma(n+\rho-j+1) \Gamma(\kappa+j+1)} \bz^{m-j} z^{n-j}  \left(  1-|z|^2\right)^{ j}  . \nonumber 
	\end{align} 
	This gives rise to \eqref{Zmnexp}.
\end{proof}

Below, we present different hypergeometric representations of $\mathcal{Z}^{\kappa,\rho}_{m,n}$ in terms of the Gauss hypergeometric function  defined on the open unit disc by  power series
$${_2F_1}\left( \begin{array}{c} a , b \\ c \end{array}\bigg | z \right) = \sum_{n=0}^{\infty }{\frac {(a)_{n}(b)_{n}}{(c)_{n}}}{\frac {z^{n}}{n!}} $$
provided that $c \ne 0, -1, -2, \cdots$.

\begin{proposition}\label{proppsiHyperg}
	The functions $\mathcal{Z}^{\kappa,\rho}_{m,n}(z,\bz)  $ are given in terms of the ${_2F_1}$ function by 
	\begin{align}\label{Zhyper} \mathcal{Z}^{\kappa,\rho}_{m,n}(z,\bz)&=(\kappa+1)_m z^{n }  \bz^{m} {_2F_1}\left( \begin{array}{c} -m , -n-\rho \\ \kappa+1 \end{array}\bigg | 1-\frac{1}{|z|^2} \right) . 
	\end{align} 
\end{proposition}

\begin{proof}
	By means of $(a)^{\underline{n}}= (-1)^n (-a)_n = \Gamma(a+1)/\Gamma(a-n+1)$ combined with $ (-1)^j(-m)_j(m-j)!=m!$, we can rewrite \eqref{Zmnexp} as
	\begin{align*}
		\mathcal{Z}^{\kappa,\rho}_{m,n}(z,\bz)  
		&=   (\kappa+1)_m z^{n }  \bz^{m} \sum_{j=0}^{m\wedge^*(n+\rho)}    	    \frac{(-m)_j(-n-\rho)_j  }{(\kappa+1)_j j!}   \left(  1-\frac{1}{|z|^2}\right)^{j} \\
		&=   (\kappa+1)_m z^{n }  \bz^{m} {_2F_1}\left( \begin{array}{c} -m , -n-\rho \\ \kappa+1 \end{array}\bigg | 1-\frac{1}{|z|^2} \right).
	\end{align*} 
\end{proof}

There are several equivalent expressions for $ \mathcal{Z}^{\kappa,\rho}_{m,n}$ in terms of the Gauss hypergeometric functions which follow from the well-known linear transformations for ${_2F_1}$. Thus, from the second and the third  ones in \cite[$\S$ 2.4., p. 47]{Magnus1966}, 
it follows
\begin{align}\label{Zhyp2}
	\mathcal{Z}^{\kappa,\rho}_{m,n}(z,\bz) = (\kappa+1)_m z^{n-m }     {_2F_1}\left( \begin{array}{c} -m , n+ \kappa+\rho+1 \\ \kappa+1 \end{array}\bigg | 1- |z|^2 \right) 	\end{align}
and
\begin{align}\label{Zhyp2d}
	\mathcal{Z}^{\kappa,\rho}_{m,n}(z,\bz) = (\kappa+1)_m z^{-\rho}\bz^{m-n-\rho }     {_2F_1}\left( \begin{array}{c} -n-\rho,  \kappa+m+1 \\ \kappa+1 \end{array}\bigg | 1- |z|^2 \right) .	\end{align}
However,  starting from \eqref{Zhyper} and applying the linear transformation  \cite[Eq. (15.8.7)]{NIST} for the limiting case one obtains 
\begin{align}
	\mathcal{Z}^{\kappa,\rho}_{m,n}(z,\bz) &=
(\kappa+\rho+n+1)_m z^{n }  \bz^{m}    {_2F_1}\left({-m,-n-\rho \atop -\kappa-\rho-n-m}\bigg | \frac{1}{|z|^2}\right).
\end{align}
The same transformation applied to \eqref{Zhyp2} gives rise to
\begin{align} \label{ZhypF}
	\mathcal{Z}^{\kappa,\rho}_{m,n}(z,\bz) &= (-n-\rho)_m
	z^{n-m }      {_2F_1}\left({-m, \kappa+n+\rho+1 \atop n+ \rho  -m+1 }\bigg | |z|^2 \right).
\end{align} 
The latter one remains valid for $\rho$ being integer and $m \leq n+\rho$.



The next result is concerned with the expression of $\mathcal{Z}^{\kappa,\rho}_{m,n}(z,\bz)  $ in terms of the the real Jacobi polynomials.  

\begin{proposition}\label{proppsiJac}
	For $\rho = 0,1, \cdots$, we have 
	\begin{align}\label{ZJacobimin}
	\mathcal{Z}^{\kappa,\rho}_{m,n}(z,\bz) 
		&=  	 	 \frac{(\kappa+1)_{m\vee (n+\rho)} (m\wedge (n+\rho))! }{(\kappa+1)_{n+\rho}}  	
		\frac{z^{n}  \bz^{m}}{|z|^{2(m\wedge (n+\rho)) }}  
		P^{(\kappa,|m-n-\rho| )}_{m\wedge(n+\rho)}(2|z|^2-1),
	\end{align}	
while when  $\rho>-1$ is non-integer  or $m= m\wedge (n+\rho)$  we have
	\begin{align}\label{ZJacobi}
		\mathcal{Z}^{\kappa,\rho}_{m,n}(z,\bz) 
		&=  	 m!  z^{n-m} P_m^{(\kappa,n-m+\rho )}(2|z|^2-1) .
	\end{align}	 
\end{proposition} 

\begin{proof}	
	Consider first the case of $\rho = 0,1, \cdots$. Starting from  \eqref{Zhyp2} and the assumption  $m \leq n+\rho$ one can make use of the facts $(-n-\rho)_m =(-1)^m (n+ \rho  -m+1)_m$, $ 
	P^{(a,b)}_m(x)= (-1)^m P^{(b,a)}_m(-x)$  and 
		\begin{align} \label{jacobi} {_2F_1}\left( \begin{array}{c} -m , m+a+b+1 \\ a+1 \end{array}\bigg | x\right) = \frac{m!}{(a+1)_m} P^{(a,b)}_m(1-2x)
	\end{align}
	to get   
\begin{align}
	\mathcal{Z}^{\kappa,\rho}_{m,n}(z,\bz)
	&= m!  z^{n-m }  
	P^{(\kappa,n+ \rho  -m)}_m(2|z|^2-1). 
\end{align} 
	%
The result corresponding to the case $m\geq n+\rho$ is immediate from the previous one using the like-symmetry relationship \eqref{Sym} (it is also immediate from \eqref{Zhyp2d}). In fact, we have 
\begin{align} 
	\mathcal{Z}^{\kappa,\rho}_{m,n}(z,\bz)
&	=   \frac{(\kappa+1)_m (n+\rho)! }{(\kappa+1)_{n+\rho}}  z^{-\rho} 	
   \bz^{m-n-\rho }  
P^{(\kappa,m   -n-\rho)}_{n+\rho}(2|z|^2-1).
\end{align} 
To conclude one observes that both expressions can rewritten in the  unified form 
\begin{align*}
	\mathcal{Z}^{\kappa,\rho}_{m,n}(z,\bz)  
& = \frac{(\kappa+1)_{m\vee (n+\rho)} (m\wedge (n+\rho))! }{(\kappa+1)_{n+\rho}}  |z|^ {|m-n-\rho| -\rho } 	
e^{i(n-m)\arg z }  
P^{(\kappa,|m-n-\rho| )}_{m\wedge(n+\rho)}(2|z|^2-1) ,
\end{align*} 
which clearly leads to \eqref{ZJacobimin}.
This can also be deduced from 
	the functions $z^\rho \mathcal{Z}^{\kappa,\rho}_{m,n}$ being the classical Zernike functions in \eqref{DiscePol} up to multiplicative constant and next making appeal to the one in  \cite{Ghanmi2010}.

	%
	%
	%
	%
	%
	%
Finally, the formula \eqref{ZJacobi} for $\rho$ being non-integer is exactly \eqref{ZJacobimin}  (when $\rho$ integer and $m\leq n+\rho$). It  can be obtained by means of  \eqref{Zhyp2} and \eqref{jacobi}
valid whenever $m= m\wedge ^* (n+\rho)$, which corresponds to $\rho>-1$ being non-integer  or also to $m= m\wedge (n+\rho)$.
A direct proof can be handled starting from the derivation formula \cite[p. 1, 1.1.2 Eq. 2]{Brychkov2008} that we can rewrite as  
	\begin{align}
		\frac{d^m}{d z^m} \left( z ^a (1-xz)^n  \right) = m! z^{a-m} (1-xz)^{b-m} P_m^{(a-m,b-m)}(1-2xz) .
	\end{align}
	Therefore, with the specification $a=n+\rho$, $b=\kappa+m$ and $x=\bz$, the expression of $\mathcal{Z}^{\kappa,\rho}_{m,n}$ in \eqref{betZpol} becomes \eqref{ZJacobi}.
\end{proof}
 
The next result concerning the zeros of $	\mathcal{Z}^{\kappa,\rho}_{m,n}$ is an immediate consequence of Proposition \ref{proppsiJac}.

\begin{corollary}\label{corzeros}
The point $0$ is a zero of $	\mathcal{Z}^{\kappa,\rho}_{m,n}$ when $\rho=0,1,2,\cdots$ if and only if $n>m$ or $m >n$ with $\rho=0$. However, when $\rho$ is non-integer, it is a zero of $	\mathcal{Z}^{\kappa,\rho}_{m,n}$ only when $n>m$. 
 The other zeros are the circles centered at the origin with radii $(1+r_{m,n}^{\kappa,\rho})/2$, where $r_{m,n}^{\kappa,\rho}$ are 
 the zeros, located at the segment $(0,1)$, of the real Jacobi polynomials 
 $P_m^{(\kappa,n-m+\rho )}(x)$ when $\rho$ is non-integer and $P_{m\wedge (n+\rho)}^{(\kappa,|n-m+\rho| )}(x)$ when $\rho$ is a non-negative integer. 
\end{corollary}


\begin{remark}\label{rembetaZJacobi}
	According to Proposition \ref{proppsiJac}, the expression of the $\beta$-restricted Zernike functions $\psi^{\gamma,\eta}_{m,n}$ in terms of the Jacobi polynomials reads 
	\begin{align}\label{psiJacobi}
		\psi^{\gamma,\eta}_{m,n}(z,\bz)
		&=	(-1)^m m!  \frac{z^{n-m}}{|z|^{2\eta}} (1-|z|^2)^{\frac{\kappa-\alpha-1}{2}}
		P_m^{(n-m+\rho,\kappa)} (1-2|z|^2),
	\end{align}
	where $\rho = \beta-2\eta>-1$ is non-integer and $\kappa= \alpha-2(\gamma+m)-1$.
\end{remark} 

We conclude this subsection by discussing the orthogonality of the considered functions.

\begin{corollary}\label{corOrthnorm} The functions $\mathcal{Z}^{\kappa,\rho}_{m,n}$ form an orthogonal system  in $L^{2,\kappa}_{\rho}(\D)$ with square norm given by 
	\begin{align} \label{normFrZ}
		\norm{\mathcal{Z}^{\kappa,\rho}_{m,n}}_{L^{2,\kappa}_{\rho}(\D)}^2 = \frac{\pi m!  (n+\rho)!\Gamma (m+\kappa +1)}{(m+n+\rho +\kappa +1)\Gamma (n+\rho +\kappa +1) } 
		%
		=: \frac{1}{\gamma^{\kappa,\rho}_{m,n}}
	\end{align}

\end{corollary} 

\begin{proof}
	 We provide explicit 
	computation only when $\rho$ being integer.
	For the case of $\rho$ non-integer one can proceeds as for $m= m\wedge (n+\rho)$ and $\rho$ is integer. Thus, let $\rho$ a fixed integer and set 
	$$	d_{m,n}^{\rho,\kappa}  :=
	\frac{(\kappa+1)_{m\vee (n+\rho)} (m\wedge (n+\rho))! }{(\kappa+1)_{n+\rho}} .$$
	Then, from Proposition \ref{proppsiJac} and the use of the polar coordinates $z=\sqrt{t} e^{i\theta}$; $0\leq t <1$, $0 \leq \theta < 2\pi$, we get
		\begin{align*}
		I_{m,n,j,k}^{\rho,\kappa} & :=
	\int_{D} \mathcal{Z}^{\kappa,\rho}_{m,n} (z,\bz) \overline{ \mathcal{Z}^{\kappa,\rho}_{j,k}(z,\bz)} |z|^{2\rho }(1-\vert z\vert^{2})^{\kappa}d\Lambda(z) \\
			&= 	\pi  d_{m,n}^{\rho,\kappa}  	d_{j,k}^{\rho,\kappa} \left(  \int_0^1 t^ {|m-n-\rho|} (1-t)^{\kappa} P^{(\kappa,|m-n-\rho| )}_{m\wedge(n+\rho)}(2t-1)  
		P^{(\kappa,|m-n-\rho| )}_{j\wedge(k+\rho)}(2t-1) 
	 dt \right) \delta_{n-m,k-j}   \\
	&= 	\frac{\pi  d_{m,n}^{\rho,\kappa}  	d_{j,k}^{\rho,\kappa}}{2^{|m-n-\rho|+\kappa+1}} \left(    \int_{-1}^1 (1+x) ^ {|m-n-\rho|}  (1-x) ^{\kappa}  P^{(\kappa,|m-n-\rho| )}_{m\wedge(n+\rho)}(x)  
	P^{(\kappa,|m-n-\rho| )}_{j\wedge(k+\rho)}(x) 
	 dx\right) \delta_{n-m,k-j} .
	\end{align*}
		Now, by the orthogonal property for the classical Jacobi polynomials \cite[p.212]{Magnus1966},
	it follows

			\begin{align*}
		I_{m,n,j,k}^{\rho,\kappa}  
		&= 	\frac{\pi  d_{m,n}^{\rho,\kappa}  	d_{j,k}^{\rho,\kappa}}{2^{|m-n-\rho|+\kappa+1}}     \norm{ P^{(\kappa,|m-n-\rho| )}_{m\wedge(n+\rho)}}^2 
		\delta_{n-m,k-j} \delta_{m\wedge(n+\rho),j\wedge(k+\rho)} 
		\\
		&= \frac{\pi m!  (n+\rho)!\Gamma (m+\kappa +1)}{(m+n+\rho +\kappa +1)\Gamma (n+\rho +\kappa +1) }\delta_{m,j} \delta_{n,k}.
\end{align*}
	This proves the orthogonality of $\mathcal{Z}^{\kappa,\rho}_{m,n}$ in the Hilbert space $L^{2,\kappa}_{\rho}(\D)$.
\end{proof}

\subsection{Poly-meromorphy} 
In analogy with the definition of the polyanalytic functions defined as those satisfying the Cauchy--Riemann equation $\partial^n/\partial \bz^n=0$ one has to suggest the following or fthe poly-meromorpy \cite[p 199]{Balk1997}.
 
 \begin{definition}
 	 \label{defMeremor}
	A complex-valued function $f$ on an open set $U$ in the complex plane is said to be poly-meromorphic of order $n$ (of first kind) if there exist certain meromorphic functions $\psi_k$; $k=0,1, \cdots,n-1$ on $U$ such that 
	$$ f(z) =   \psi_{0}(z)+  \bz  \psi_{1}(z) + \cdots + \bz^{n-1} \psi_{n-1}(z) .$$
\end{definition} 

The main result in this subsection  discusses the regularity of the considered fractional Zernike functions. 

\begin{theorem} \label{thmAnaMeremor}
	The fractional Zernike functions $	\mathcal{Z}^{\kappa,\rho}_{m,n}$ are polynomials in $z$ and $\bz$ if and only if $\rho=0$ or $m\leq n$. Alternatively, they are poly-meromorphic functions of order $m$ with $0$ as unique pole. Its order of multiplicity is given by 
	$Ord^{\kappa,\rho}_{m,n}= \rho$ for $m > n+\rho $ when $\rho = 1,2, \cdots$, and by $Ord^{\kappa,\rho}_{m,n}= m-n$ for $m>n$ when $\rho >-1$ is non-integer, or when $n < m \leq n+\rho$ with $\rho=1,2, \cdots$.
\end{theorem}

\begin{proof}  
		Set 
	\begin{align*}
		c^{\kappa,\rho}_{m,n,j}&:=  \frac{(-1)^{j}  m! \Gamma(n+\rho+1)\Gamma(\kappa+m+1)  }{j! (m-j)!\Gamma(n+\rho-j+1) \Gamma(\kappa+j+1) }  
	\end{align*}
and for $p<q \leq m\wedge^*(n+\rho)$  consider the quantities $$S^{\kappa,\rho,m,n}_{q,p} := R^{\kappa,\rho,m,n}_{q}  - X^{p-q}R^{\kappa,\rho,m,n}_{p},$$ where $R^{\kappa,\rho,m,n}_{p} $ is the polynomial of degree less or equal to $p$  given by
\begin{align}\label{Spol} 
	R^{\kappa,\rho,m,n}_{p}(X) 
	& = \sum_{k=0}^{p} \left( \sum_{j=0}^{p-k}  (-1)^j \frac{(j+k)!}{j!k!} c^{\kappa,\rho}_{m,n,j+k}  \right)  X^{p-k }.
\end{align}
Notice for instance that its constant coefficients is given by  
$S^{\kappa,\rho,m,n}_{q,p}(0) = R^{\kappa,\rho,m,n}_{q}(0)  = c^{\kappa,\rho}_{m,n,q} .$
 Thus, starting from \eqref{Zmnexp} we can rewrite $	\mathcal{Z}^{\kappa,\rho}_{m,n}(z,\bz)$ as
	\begin{align} 
		\mathcal{Z}^{\kappa,\rho}_{m,n}(z,\bz)
		&= z^{n-[m\wedge^*(n+\rho)] }  \bz^{m-[m\wedge^*(n+\rho)]} R^{\kappa,\rho,m,n}_{m\wedge^*(n+\rho)}(|z|^2) .	\label{mermer1}
 \end{align}
But, since $c^{\kappa,\rho}_{m,n,p} \ne 0$, it becomes clear from \eqref{mermer1} that the functions $	\mathcal{Z}^{\kappa,\rho}_{m,n}$ are polynomials if and only if  $\rho=0$ or $m\leq n$ independently of $\rho>-1$ being integer or not. 
 This assertion is also immediate from Proposition \ref{proppsiJac}. 
	Next, using the fact that 
	$R^{\kappa,\rho,m,n}_{q} = X^{q-p}R^{\kappa,\rho,m,n}_{p}  + S^{\kappa,\rho,m,n}_{q,p} $ we obtain 
		\begin{align} 
		\mathcal{Z}^{\kappa,\rho}_{m,n}(z,\bz) 
		&=    \bz^{m-n}  R^{\kappa,\rho,m,n}_{n}(|z|^2)	 +  z^{n-[m\wedge^*(n+\rho)] }  \bz^{m-[m\wedge^*(n+\rho)]} S^{\kappa,\rho,m,n}_{m\wedge^*(n+\rho),n}(|z|^2) . \label{mermer2}
	\end{align}
A meticulous study of the different possible cases  of $m$ compared to $n$ and $n+\rho$ for  given $\rho>-1$  
	 leads to  
	\begin{align*} 
	\mathcal{Z}^{\kappa,\rho}_{m,n}(z,\bz)
	=
	\left\{ 
	\begin{array}{lllll}	
		\displaystyle	
		z^{n-m} R^{\kappa,\rho,m,n}_{m}(|z|^2) &  \mbox{if } m\leq n ; \,  \rho >-1  \\
			\displaystyle	
			\bz^{m-n} R^{\kappa,\rho,m,n}_{n}(|z|^2) &  \mbox{if }  m>n; \, \rho =0 \\
		\displaystyle	\bz^{m-n} R^{\kappa,\rho,m}_{n}(|z|^2) + \frac{1}{z^{m-n}}  S^{\kappa,\rho,m,n}_{m}(|z|^2)  &  \mbox{if }   m>n ; \, \rho \mbox{ non-integer}\\ &  \mbox{or } n < m \leq n+\rho; \, \rho=1,2,\cdots \\
		\displaystyle
		\bz^{m-n} R^{\kappa,\rho,m,n}_{n}(|z|^2) + \frac{\bz^{m-(n+\rho)}}{z^\rho} S^{\kappa,\rho,m,n}_{n+\rho}(|z|^2)  &  \mbox{if }  m > n+\rho ; \, \rho=1,2, \cdots.
	\end{array} \right.
\end{align*} 
Moreover, this reveals that the regular part of $\mathcal{Z}^{\kappa,\rho}_{m,n}$ is always a polyanalytic function of order $m$ and anti-polyanalytic of order $n$. It is  given by  
	\begin{align}\label{Tpolmin}
		R^{\kappa,\rho}_{m, n}(|z|^2)&= \left\{ \begin{array}{ll}
			z^{n-m}	 R^{\kappa,\rho,m,n}_{m}(|z|^2), &  m\leq n \\
			\bz^{m-n}	 R^{\kappa,\rho,m,n}_{n}(|z|^2), &   m \geq n
		\end{array} \right. 
	= z^{n-m\wedge n} \bz^{m-m\wedge n}	 R^{\kappa,\rho,m,n}_{m\wedge n}(|z|^2)  .
	\end{align}
However,  in general $\mathcal{Z}^{\kappa,\rho}_{m,n}$ are poly-meremorphic with $0$ as the unique pole for $\rho \ne0$ and $m>n$. 
Thus, for $m > n+\rho$ with $\rho=1,2, \cdots$ the singular part is clearly given by 
$ z^{-\rho} \bz^{m-(n+\rho)} S^{\kappa,\rho,m,n}_{n+\rho}(|z|^2).$ 
It is given by 
$z^{n-m} S^{\kappa,\rho,m,n}_{m}(|z|^2)$ whenever $ n < m $ and $\rho>-1$ non-integer or $n < m \leq n+\rho$ when $ \rho=1,2,\cdots$.
 Therefore, it becomes clear that the multiplicity of the singularity is given by
	\begin{align*} 
Ord^{\kappa,\rho}_{m,n}
&	=
	\left\{ 
	\begin{array}{ll}	
		\displaystyle
		 \rho    &  \mbox{if }  m > n+\rho ; \, \rho=1,2, \cdots\\
		 \displaystyle   {m-n}   &  \mbox{if }  n < m ; \, \rho \mbox{ non-integer}\\ &  \mbox{or } n < m \leq n+\rho; \, \rho=1,2,\cdots  .
	\end{array} \right.
\end{align*}
This completes the proof.
\end{proof}


\begin{remark} The obtained result can justifies somehow the appellation of $\mathcal{Z}^{\kappa_m,\rho}_{m,n}$ by fractional Zernike functions. 
\end{remark}

\subsection{Differential equations.}

In this subsection we are concerned with some second order differential equations satisfied by the fractional Zernike functions.

\begin{theorem}\label{thmDiffeq}
	Let  $m_\rho=m+\rho $. Then, the function  $\mathcal{Z}^{\kappa,\rho}_{m,n}$ is solution of	
	\begin{align*}
			z^2 (1-|z|^2) \frac{\partial^2}{\partial z^2}+ \left( (m_\rho-n+1) -(\kappa+m_\rho-n+2)|z|^2\right) z \frac{\partial}{\partial z} +n(\kappa+m_\rho+1)|z|^2  =\rho(n-m) .
		\end{align*} 
\end{theorem}

\begin{proof}
	Consider the first order differential operator
		\begin{align}\label{craopeprof1} \nabla^{\kappa,\rho} (f)(z)= - \left( (1-|z|^2) \frac{\partial}{\partial z}  -\mathcal{Z}^{\kappa,\rho}_{1,0}(z,\bz)\right)   (f)(z).
	\end{align} 
Also, for varying $j=1,2,\cdots ,m$ we set
	\begin{align}\label{craopeprof}
		\nabla^{\kappa,\rho}_{j}(f) &:= -z^{-\rho} (1-|z|^2)^{-\kappa-j+1}\frac{\partial}{\partial z} \big( z^{\rho}(1-|z|^2)^{\kappa+j}f\big),
	\end{align}
	 so that $\nabla^{\kappa,\rho} (f)(z)= \nabla^{\kappa,\rho}_1 (f)(z)$.
Successive application of 	
	$\nabla^{\kappa,\rho}_{j}$ 
	leads to the 	
	operator $\widetilde{\nabla}^{\kappa,\rho}_{m} := \nabla^{\kappa,\rho}_{1}\circ \nabla^{\kappa,\rho}_{2}\circ \cdots \circ\nabla^{\kappa,\rho}_{m}  $ satisfying 
	$\widetilde{\nabla}^{\kappa,\rho}_{m+1} 
	=
	\nabla^{\kappa,\rho}_{1}\circ \widetilde{\nabla}^{\kappa+1,\rho}_{m}  $ since $ \nabla^{\kappa,\rho}_{j+1}= \nabla^{\kappa+1,\rho}_{j}$.
	It is explicitly given by
	\begin{equation}\label{Ac}
		\widetilde{\nabla}^{\kappa,\rho}_{m}(f)(z)=(-1)^m z^{-\rho} (1-|z|^2)^{-\kappa}\frac{\partial^m}{\partial z^m} \big( z^{\rho}(1-|z|^2)^{\kappa+m}f\big).
	\end{equation}
	Thus, in view of \eqref{betZpol} it is clear that $
	\widetilde{\nabla}^{\kappa,\rho}_{m} (e_n)(z) = \mathcal{Z}^{\kappa,\rho}_{m,n}(z,\bz)$  for $e_n(z):= z^n$, and therefore 
		\begin{align} \label{nablamm}
		\nabla^{\kappa,\rho}( \mathcal{Z}^{\kappa+1,\rho}_{m,n}(z,\bz)) = \nabla^{\kappa,\rho}_{1}\left(\nabla^{\kappa+1,\rho}_{1}\circ \nabla^{\kappa+1,\rho}_{2}\circ \cdots \circ\nabla^{\kappa+1,\rho}_{m}\right)(e_n)   =\mathcal{Z}^{\kappa,\rho}_{m+1,n}(z,\bz) .
	\end{align} 
	%
Now associated with the Euler differential operator $E_z= z {\partial}/{\partial z}$ and the constant $c^{\kappa,\rho}_{m,n}=m(n+\rho+\kappa+1) $, we define the first order differential operator
$$  D^{\kappa,\rho}_{m,n}=   \frac{1}{c^{\kappa,\rho}_{m,n} \bz} \left( E_z -(n-m) .\right)  = 
\frac{1}{c^{\kappa,\rho}_{m,n}}  \left( \frac{z}{\bz} \frac{\partial}{\partial{z}} - \frac{(n-m)}{\bz}\right) .$$
Hence making use of \eqref{ZJacobi} combined with 
	the differentiation formula of Jacobi polynomials in \cite[p.213 ]{Magnus1966} we obtain the identity
		\begin{equation}\label{Oop}
		D^{\kappa,\rho}_{m,n}(\mathcal{Z}^{\kappa,\rho}_{m,n} )=\mathcal{Z}^{\kappa+1,\rho}_{m-1,n}.
	\end{equation} 
Therefore, from \eqref{nablamm}  and \eqref{Oop} it is immediate that  
$ 
\nabla^{\kappa,\rho} \circ	D^{\kappa,\rho}_{m,n}(\mathcal{Z}^{\kappa,\rho}_{m,n}(z,\bz) )
= \mathcal{Z}^{\kappa,\rho}_{m,n} (z,\bz)$.
	A direct computation shows that $ -c^{\kappa,\rho}_{m,n} \bz \nabla^{\kappa,\rho}\circ D^{\kappa,\rho}_{m,n}$ 
	is       given  by
	\begin{align*}
	   z (1-|z|^2) \frac{\partial^2}{\partial z^2} + \left( [m_\rho-n +1 ] - [\kappa+m_\rho-n +2] |z|^2 \right)  \frac{\partial}{\partial z} + (m-n)\left( \frac{\rho}{z}-[\kappa+\rho+1] \bz\right)  
\end{align*}
	This shows that the fractional Zernike function $\mathcal{Z}^{\kappa,\rho}_{m,n}$ satisfy the desired differential equation. 
		%
	%
	%
	%
	%
	%
	%
	%
	%
\end{proof}

\begin{remark}\label{lemHhaj}
	In view of \eqref{nablamm} and \eqref{Oop} the considered operators $\nabla^{\kappa,\rho}$ and $D^{\kappa,\rho}_{m,n}$ appear as creation and annihilation operators for the fractional Zernike functions.
\end{remark}

\begin{remark} Let $(P_{j}f)(z)=z^j f(z)$.
	Then, the commutation relation 
	$P_j \circ \nabla^{\kappa,\rho}_{m}\circ P_j^{-1}=\nabla^{\kappa,\rho-j}_{m} $
	holds for all $z \in \mathbb{D}^{*}$. 	
	This follows by observing that from \eqref{betZpol}, we have
	\begin{align}\label{Multz}
		z^j\mathcal{Z}^{\kappa,\rho}_{m,n}(z,\bz)=\mathcal{Z}^{\kappa,\rho-j}_{m,n+j}(z,\bz).
	\end{align}
\end{remark}

The following can be proved using the close connection of $\mathcal{Z}^{\kappa_m,\rho}_{m,n}$ to the $\beta$-restricted Zernike functions studied in Section 2. 

\begin{theorem}\label{thmDiffeqs} 
	For given fixed nonnegative integer $m$ and reals $\alpha>-1$ and $\gamma$ such that $\kappa_m = \alpha-2(\gamma+m)-1$, the fractional Zernike functions $\mathcal{Z}^{\kappa_m,\rho}_{m,n}$ for varying $n$ are eigenfunctions of 
	\begin{align}\label{2odeq2}
		-(1-|z|^2)^2 \partial\overline{\partial}-  (1-|z|^2)\left(m  E - H_{\kappa_m+m+1}^{\rho}(z) \overline{E}  \right)  
		+ m H_{\kappa_m+m+1}^{ \rho}(z)  |z|^2   
	\end{align}
	with $m(\kappa_m+m+1)$ as corresponding eigenvalue.
\end{theorem}

\begin{proof} 
	For the proof observe that the fractional Zernike functions 
	$\mathcal{Z}^{\kappa_m,\rho}_{m,n}$ are closely connected to the $\beta$-restricted Zernike functions $ \psi^{\gamma,\eta}_{m,n}(z,\bz)$ by 
	\eqref{Zernike0} for every fixed nonnegative integer $m$. The latter ones are eigenfunctions of the hamiltonian $
	\mathcal{L}_{\gamma,\eta}^{\alpha,\beta,+}$ in \eqref{LpLm} with $ E^{\gamma,\alpha}_m = 	(m+1)(\alpha-2\gamma-m)$ as corresponding eigenvalue (see $(i)$ in Theorem \ref{Mth3}). 
	The key observation to conclude is that the operators
	$
	\mathcal{L}_{\gamma,\eta}^{\alpha,\beta,+} 
	$ and $
	\mathcal{L}_{\gamma+a,\eta+b}^{\alpha+2a,\beta+2b,+} 
	$ are unitary equivalent for arbitrary reals $a$ and $b$. 
	More precisely, since $ A^{*_{\alpha,\beta}}_{\gamma,\eta} (h_{a,b} f) 
	= h_{a,b} A^{*_{\alpha+2a,\beta+2b}}_{\gamma+a,\eta+b}(f)$
	and 
	$A_{\gamma,\eta} (h_{a,b} f) = h_{a,b} A_{\gamma+a,\eta+b} (f)$, we obtains 
	\begin{align*} \mathcal{L}_{\gamma,\eta}^{\alpha,\beta,+} (h_{a,b} f)& =
		A_{\gamma,\eta} A^{*_{\alpha,\beta}}_{\gamma,\eta} (h_{a,b} f)
		= 
		h_{a,b} A_{\gamma+a,\eta+b}  A^{*_{\alpha+2a,\beta+2b}}_{\gamma+a,\eta+b}(f)
		=
		h_{a,b} \mathcal{L}_{\gamma+a,\eta+b}^{\alpha+2a,\beta+2b,+} (f)   
		.	\end{align*}
	Subsequently, for $\kappa_m = \alpha-2(\gamma+m)-1$, $\rho = \beta-2\eta$, $b= -\eta$ and $a= {(\kappa_m-\alpha-1)}/{2} =-(\gamma+m+1)$, the fractional Zernike function  $\mathcal{Z}^{\kappa_m,\rho}_{m,n}$ satisfies 
	$$ \mathcal{L}_{-m-1,0}^{\kappa_m-1,\rho,+}  \mathcal{Z}^{\kappa_m,\rho}_{m,n} = E^{\gamma,\alpha}_m  \mathcal{Z}^{\kappa_m,\rho}_{m,n}.$$ 
	But from Lemma \ref{th1}, it is clear that the second order partial differential equation in \eqref{2odeq2} 
	is exactly $\mathcal{L}_{-m-1,0}^{\kappa_m-1,\rho,+} - (\kappa_m+m+1)$.
\end{proof}

\begin{corollary} The Zernike polynomials $\mathcal{Z}^{(-1)}_{m,n}$, corresponding to the limit case of $\kappa_m=-1$ and fixed $m$, are harmonic functions for the Laplacian
	$$\left\{  	(1-|z|^2) \partial\overline{\partial}+ m\left( E - \overline{E}  \right) - m^2 \right\}
	\mathcal{Z}^{-1}_{m,n} .$$ 
\end{corollary}

\begin{proof}
	This readily follows by specifying $\rho=0$  in Theorem \ref{thmDiffeqs} and choosing $ \alpha$ and $\gamma$ such that $m = (\alpha/2)-\gamma$. 
	Indeed, in this case we have $\kappa_m +m +1=m$ and the left hand side of \eqref{2odeq2} reduces further to  the Landau Hamiltonian  
	$(1-|z|^2)\left\{  	(1-|z|^2) \partial\overline{\partial} + m\left( E -\overline{E}  \right)  \right\} + m^2 |z|^2
	$
	with quantized constant magnetic field of  magnitude $m$.  
\end{proof}

\subsection{Recurrence and operational formulas.}

From the three terms recurrence formula in \cite[p: 213]{Magnus1966} for the Jacobi polynomials one can deduces 
\begin{align*}
	A_{m,b}z^2\mathcal{Z}^{\kappa,\rho}_{m,n}(z,\bz) +
	B_{m,b} z\mathcal{Z}^{\kappa,\rho-1}_{m-1,n}(z,\bz) +
	C_{m,b}\mathcal{Z}^{\kappa,\rho-2}_{m-2,n}(z,\bz) =0 
\end{align*}
valid for all $m=2,3,4, \cdots$, where we have set $A_{m,b}:= (b-m)(b+2)$, $B_{m,b}:= b(b-1)(b-\kappa-1)$  and $C_{m,b}:= b(m-1)(\kappa+m-1)(b-\kappa-m)$ 
with $b=\kappa+n+m+\rho$. However, starting from the Rodriguez formula for the fractional Zernike functions by rewriting it in the form
$$ \mathcal{Z}^{\kappa,\rho}_{m,n}= (-1)^m z^{-\rho}(1-|z|^2)^{-\kappa}\partial_z^{m-1}\left( \partial_z (z^{n+\rho}(1-|z|^2)^{\kappa+m})\right) ,$$  
one derives the recurrence formula
\begin{align}\label{recF1}
	\mathcal{Z}^{\kappa,\rho}_{m,n}(z,\bz)= (\kappa+m)\bz \mathcal{Z}^{\kappa,\rho}_{m-1,n}(z,\bz) -(n+\rho)(1-|z|^2)\mathcal{Z}^{\kappa+1,\rho}_{m-1,n-1}(z,\bz).
\end{align}
	But, by means of \eqref{Multz} we can rewrite  the recurrence formula   \eqref{recF1} as
 \begin{align}\label{recF11}
 	z	\mathcal{Z}^{\kappa,\rho}_{m,n}(z,\bz)= (\kappa+m)\bz  \mathcal{Z}^{\kappa,\rho-1}_{m-1,n+1}(z,\bz) -(n+\rho)(1-|z|^2)\mathcal{Z}^{\kappa+1,\rho-1}_{m-1,n}(z,\bz).
 \end{align}
Moreover, we can prove the following 
	\begin{align}\label{recF2}
	\mathcal{Z}^{\kappa-1,\rho-1}_{m+1,n+1}&=[\kappa|z|^2+(m-n-\rho)(1-|z|^2)]  \mathcal{Z}^{\kappa,\rho}_{m,n}-m(n+\rho+\kappa+1)\overline{z}(1-|z|^2)\mathcal{Z}^{\kappa+1,\rho}_{m-1,n}.
\end{align}
Indeed, it follows from the use \eqref{Oop}, leading to  
$$z \frac{\partial }{\partial_z }(\mathcal{Z}^{\kappa,\rho}_{m,n})-(n-m)\mathcal{Z}^{\kappa,\rho}_{m,n}=\overline{z}m(n+\rho+\kappa+1)\mathcal{Z}^{\kappa+1,\rho}_{m-1,n},$$
combined with the  derivation formula 
$$ (1-|z|^2) \frac{\partial }{\partial_z }(\mathcal{Z}^{\kappa,\rho}_{m,n})=-\rho(1-|z|^2)\mathcal{Z}^{\kappa,\rho+1}_{m,n-1}+\kappa\overline{z}\mathcal{Z}^{\kappa,\rho}_{m,n}-\mathcal{Z}^{\kappa-1,\rho}_{m+1,n} .$$
The latter one follows from the Rodrigues Formula.\eqref{betZpol}
In the sequel, we obtain non-trivial recurrence formulas of Nielsen type for the fractional Zernike functions. This follows as specific cases of the so-called Burchnall representation type formulas for the fractional Zernike functions. 
To the exact statement we let $\bold{a} ^{\kappa,\rho,m,n}_{q,j,\ell,k}$ and $\bold{b}^{\kappa,\rho}_{m,n,j,k}$ respectively, be the constants given by 
\begin{align}\label{cst1}
	 \bold{a} ^{\kappa,\rho,m,n}_{q,j,\ell,k}:= \varepsilon^*_{\rho,m-j} \frac{(-1)^{m+j+\ell}  m!q! \Gamma(\rho+1) \Gamma(\kappa+m+1) }{\ell!(m-j)!(j-\ell)!(q-\ell)! \Gamma(\rho-m+j+1) \Gamma(\kappa+m+n+1)}   
\end{align}
and
\begin{align}\label{cst2} \bold{b}^{\kappa,\rho,m,n}_{j,k} := \frac{(-1)^{j+k} m!n! \Gamma(\kappa+m+n+1)}{j!k!(m-j)!(n-k)! \Gamma(\kappa+m+k+1)}.
\end{align}

\begin{proposition}
	Let $\kappa,\rho,m$ and $n$ be as above. Let $p$ be a nonnegative integer and $u$ a real such that $u\geq \max(-\kappa,-1)$. Then, we have 
	\begin{align}
		\mathcal{Z}^{\kappa,\rho}_{m, n+q}(z,\bz) 
		&=   z^{q} \sum_{j=0}^{m} \sum_{\ell=0}^{j\wedge q}  
		  \bold{a} ^{\kappa,\rho,m,n}_{q,j,\ell,k}    \left( \frac{1-|z|^2}{z}\right)^{m+\ell-j}
		 \mathcal{Z}^{\kappa+m+\ell-j}_{j-\ell, n}(z,\bz),
		\label{Burch11a}
			\end{align} 
		\begin{align}
		\mathcal{Z}^{\kappa,\rho}_{m, n+q}(z,\bz) 
		&= \frac{m!q!\Gamma(\kappa+m
			+1)}{\Gamma(\kappa+m +n+1)} z^{q}\sum_{j=0}^{m\wedge q}\sum_{k=0}^{n}   \frac{(-1)^{j}   }{j!(m-j)!(q-j)!}  \left( \frac{1-|z|^2}{z}\right) ^{j}
		\mathcal{Z}^{\kappa+j,\rho}_{m-j, n}(z,\bz)
		 \label{Burch22b}
			\end{align}
		and 
	\begin{align} 	  \mathcal{Z}^{\kappa+u,\rho}_{m, n}(z,\bz)
		&=  \frac{\Gamma(\kappa+u+m
			+1)}{ \Gamma(\kappa+u+m
			+n+1)}
		 \sum_{j=0}^{m}  \sum_{k=0}^{n}  (-1)^{j+k} \bold{b}^{\kappa,\rho,m,n}_{j,k}
		 \mathcal{Z}^{u-j-k}_{j,k}(z,\bz)
		. 
		\label{Burch33b}
	\end{align} 
\end{proposition}

\begin{proof}
	Considering the operator 
	\begin{align}\label{Ab}
		B^{\kappa,\rho}_{m,n}(f)&= \frac{\partial^m}{\partial{z}^m} \left( z^{\rho}\frac{\partial^n}{\partial{\overline{z}}^n}\left( (1-|z|^2)^{\kappa+m+n}f\right) \right) , 
	\end{align} 
	for every sufficiently differentiable function $f$. Making use of the Leibnitz formula applied  from outside to inside we arrive at the Burchnall type formula  
	\begin{align}
		B^{\kappa,\rho}_{m,n}(f)
		&=\sum_{j=0}^{m}\binom{m}{j} \frac{\partial^{m-j}}{\partial{z^{m-j}}}(z^{\rho} ) \frac{\partial^j}{\partial{z^j}} \left(  \frac{\partial^{n}}{\partial{\bz^{n}}} [(1-|z|^2)^{\kappa+m+n} f]\right) \label{Burch1}   
		\\
				&=z^\rho (1-|z|^2)^{\kappa+m }
		\sum_{j=0}^{m} \sum_{\ell=0}^{j} \sum_{k=0}^{n} 
		\bold{a} ^{\kappa,\rho}_{m,n,j,\ell,k} z^{j-m} (1-|z|^2)^{k+\ell-j}
		\mathcal{Z}^{\kappa+m+k+\ell-j}_{j-\ell, n-k}(z,\bz)
		\frac{\partial^{\ell+k}}{\partial z^{\ell}\partial \bz^{k}}(f) ,  \nonumber	
	\end{align} 
where the involved constant is given by	$$  \bold{a} ^{\kappa,\rho}_{m,n,j,\ell,k} :=  (-1)^{m+n+k}  \frac{ n! (q-\ell)!   \Gamma(\kappa+m+n+1)}{q! k! (n-k)!   \Gamma(\kappa+m+1) }  \bold{a} ^{\kappa,\rho,m,n}_{q,j,\ell}.$$
Similarly we get  (by Leibnitz formula from inside to outside)
	\begin{align}
		B^{\kappa,\rho}_{m,n}(f)
		&=\frac{\partial^{m}}{\partial{z^{m}}}\left( z^\rho\left[ \sum_{k=0}^{n} \binom{n}{k} \frac{\partial^{n-k}}{\partial{\bz^{n-k}}}\left( (1-|z|^2)^{\kappa+m+n}\right) 
		\frac{\partial^{k}}{\partial{\bz^{k}}}  (f)\right]\right) \nonumber
		\\
		&=(-1)^{m+n} z^\rho 
		 \sum_{j=0}^{m}\sum_{k=0}^{n} \bold{b}^{\kappa,\rho}_{m,n,j,k} 
		  (1-|z|^2)^{\kappa+j+k}
		\mathcal{Z}^{\kappa+j+k,\rho}_{m-j, n-k}(z,\bz)
		\frac{\partial^{j+k}}{\partial z^j\partial\bz^k}  (f) .  \label{Burch22}
	\end{align} 
%
	Therefore, since the action of $	B^{\kappa,\rho}_{m,n}$ on the specific case of $f= e_q=z^q$ reduces to 
	\begin{align*} 
		B^{\kappa,\rho}_{m,n}(z^q)
		&= (-1)^{m+n} (\kappa+m+1)_n  z^\rho (1-|z|^2)^{\kappa} \mathcal{Z}^{\kappa,\rho}_{m, n+q}(z,\bz) ,
	\end{align*} 
	we obtain \eqref{Burch11a} (resp. \eqref{Burch22b}) from \eqref{Burch1} (resp. \eqref{Burch22}). The identity \eqref{Burch33b} follows from \eqref{Burch22} by  considering the particular case of   $f(z)=(1-|z|^2)^{u}$ and observing that  $	B^{\kappa,\rho}_{m,n}((1-|z|^2)^{u}g)   =B^{\kappa+u,\rho}_{m,n}(g)$.   
\end{proof}

\begin{remark}
The Burchnall representation in \eqref{Burch22} for $f$ being an holomorphic function simply reads 
\begin{align}
	B^{\kappa,\rho}_{m,n}(f) 
	&=(-1)^{m+n} z^\rho h^\kappa  \sum_{j=0}^{m}\bold{b}^{\kappa,\rho}_{m,n,j,0}  (1-|z|^2)^{j}
	\mathcal{Z}^{\kappa+j,\rho}_{m-j, n}(z,\bz)
	\frac{\partial^{j} (f)}{\partial z^j } .  
\end{align} 
	Analog representation for arbitrary $f$ (not necessary holomorphic) can be developed using the differential operator
	\begin{equation}\label{Aa}
		A^{\rho,\kappa}_{m,n}(f)=\frac{\partial^m}{\partial{z}^m} \left( z^{n+\rho}(1-|z|^2)^{\kappa+m}f\right) .
	\end{equation}
	More precisely, one obtains
	\begin{align}  A^{\rho,\kappa}_{m,n}(f)
		&= (-1)^{m} 	m!	z^\rho h^{\kappa} \sum_{j=0}^{m}\frac{(-1)^j(1-|z|^2)^j }{j! (m-j)!}  \mathcal{Z}^{\kappa+j,\rho}_{m-j,n}(z,\bz) \frac{\partial^{j} f}{\partial{z^{j}}} . \label{BurchA} 
	\end{align} 
	
\end{remark}
\subsection{Generating and bilinear generating functions.}

The aim here is to obtain some generating and bilinear generating functions for the fractional Zernike functions. First, it is worth noting that from Proposition \ref{proppsiJac}
and making use of the generating function for the Jacobi polynomials in \cite[p: 213]{Magnus1966}  we obtain the generating function 
$$\sum_{m=0}^{+\infty}\frac{u^m}{m!} \mathcal{Z}^{\kappa,\rho}_{m,n}(z,\bz)=2^n z^{2n+1-m+\rho+\kappa} \frac{(z-u+R(u,z))^{m-n-\rho}(z+u+R(u,z))^{-\kappa}}{R(u,z)}.$$ 
Here $R(u,z)=1$ for $u=0$ and 
$R(u,z)=(z ^2 +2uz(1-2|z|^2)+z^2(1-2|z|^2)^2)^{1/2}$ when $u \neq 0$.  

The next result gives the expression of special bilinear generating functions
as derivative of the confluent and Gauss hypergeometric functions by means of the partial differential operator 
\begin{align*}
	R_{m}^{\kappa,\rho}f (z) &= \frac{1}{ ( z\bw)^{\rho}(1-|z|^2)^{\kappa} (1-|w|^2)^{\kappa} } 
	\frac{\partial^{2m}}{\partial z^m\partial \bw^m}\left( 
	 (z\bw)^{\rho}  (1-|z|^2)^{\kappa+m} (1-|w|^2)^{\kappa+m}
	f \right)(z)
\end{align*}
for sufficiently differential function $f$.

\begin{proposition} 
	We have 	
	\begin{align}\label{Bilin11}
		\sum_{n=0}^{+\infty} \frac{(a)_n}{n! (c)_n}  \mathcal{Z}^{\kappa,\rho}_{m,n}(z,\bz) \overline{\mathcal{Z}^{\kappa,\rho}_{m,n}(w,\bw)} &= 	R_{m}^{\kappa,\rho} \left(  
		{_1F_1}\left( \begin{array}{c} a  \\ c \end{array}\bigg | z\bw \right) \right)
	\end{align}
	and
	\begin{align}\label{Bilin21}
		\sum_{n=0}^{+\infty}
		\frac{(a)_n (b)_n}{n! (c)_n} \mathcal{Z}^{\kappa,\rho}_{m,n}(z,\bz) \overline{\mathcal{Z}^{\kappa,\rho}_{m,n}(w,\bw)} &=	R_{m}^{\kappa,\rho} \left(  
		{_2F_1}\left( \begin{array}{c} a , b \\ c \end{array}\bigg | z\bw \right) \right).
	\end{align}
\end{proposition}



\begin{proof}
	This readily follows by means of the Rodrigues' formula for  $\mathcal{Z}^{\kappa,\rho}_{m,n}(z,\bz)$. Indeed we can rewrite the left-hand side in \eqref{Bilin11} as
	\begin{align*}
		\frac{1}{ ( z\bw)^{\rho}(1-|z|^2)^{\kappa} (1-|w|^2)^{\kappa} } 
		\frac{\partial^{2m}}{\partial z^m \partial \bw^m}  \left(  (z\bw)^{\rho} [(1-|z|^2)(1-|w|^2)]^{\kappa+m} \sum_{n=0}^{+\infty}
		\frac{(a)_n}{(c)_n} \frac{z^n}{n!}\right) ,
	\end{align*}
	which reduces further to 
	\eqref{Bilin11}. 
	The formula \eqref{Bilin21} follows in a similar way. 
\end{proof}

\begin{remark} \label{remRepKer}
	For the special values of $a=1$, $b=\kappa+\rho+1$ and $c=\rho+1$ with $\rho =\beta-2\eta$ and $\kappa=\kappa_m=\alpha-2(\gamma+m) -1$,  the quantity ${n!(c)_n}/{(a)_n(b)_n} $ reduces to be the square norm of $\psi^{\gamma,\eta}_{m,n}$ in \eqref{normZernike} up to a multiplicative constant $d^{\kappa,\rho}_m$ independent of $n$.
	Thus, formula in \eqref{Bilin21} leads to the reproducing kernel  	\begin{align*} 
		K_{\gamma,\eta,m}^{\alpha,\beta}(z,w) &=
		\sum_{n=0}^{+\infty}
		\frac{\psi^{\gamma,\eta}_{m,n}(z,\bz)  \overline{\psi^{\gamma,\eta}_{m,n}(w,\bw) }}{\norm{\psi^{\gamma,\eta}_{m,n}}_{\alpha,\beta}^2} 
	\end{align*}
	of the $m$-th generalized $\beta$-modified Bergman space introduced in Remark \ref{remNoyOrth}. In fact, we have 
	\begin{align*} 
			K_{\gamma,\eta,m}^{\alpha,\beta}(z,w) 
		& = d^{\kappa,\rho}_m  \frac{ [(1-|z|^2)(1-|w|^2)]^{\gamma+m} }{|zw|^{2\eta}} 	R_{m}^{\kappa,\rho} \left(  
		{_2F_1}\left( \begin{array}{c} a , b \\ c \end{array}\bigg | z\bw \right) \right).
	\end{align*}
	A closed formula for $	K_{\gamma,\eta,m}^{\alpha,\beta}(z,w)$ needs further investigation.
\end{remark}
Below, we prove a bilinear generating function for the fractional Zernike function that looks like the Hardy--Hille formula for the generalized Laguerre polynomials. Thus, we deal with
\begin{align}\label{expHH} G^{\kappa,\rho}_n(z,w|t) &:= 	\sum_{m=0}^{+\infty}\frac{t^m \mathcal{Z}^{\kappa,\rho}_{m,n}(z,\bz)\mathcal{Z}^{\kappa,\rho}_{m,n}(\overline{w},w)}{m!(\kappa+1)_m}.
\end{align}

\begin{proposition}\label{BilnGenK2} For every  sufficiently small $t$, the closed expression of 
$G^{\kappa,\rho}_n(z,w|t)$ is given
	\begin{align*}
		G^{\kappa,\rho}_n(z,w|t) 
			&= \frac{(z+tw)^{n+\rho} (\bw + t\bz)^{n+\rho}}{(z\overline{w})^{\rho}(1+t\bz w)^{2(n+\rho)+\kappa+1)}} 	{_2F_1}\left(  \begin{array}{c}-n-\rho ,-n -\rho \\ \kappa+1 \end{array}\bigg | -\frac{t(1-|z|^2)(1-|w|^2)}{(z+tw)(\bw+t\bz)}\right) .
	\end{align*}
\end{proposition}

\begin{proof}
	Using the  hypergeometric representation (Proposition \ref{proppsiJac}) we can rewrite $G^{\kappa,\rho}_n(z,w|t)$ as 
	\begin{align*}
		G^{\kappa,\rho}_n(z,w|t) & = (z\bw)^{n}\sum_{m=0}^{+\infty}\frac{(\kappa+1)_m}{m!} (t\overline{z}w )^m \\& 
		{_2F_1}\left(  \begin{array}{c}-m,-n-\rho \\ \kappa+1 \end{array}\bigg | 1-\frac{1}{|z|^2}\right) 
		 {_2F_1}\left(  \begin{array}{c}-m,-n-\rho \\ \kappa+1 \end{array}\bigg | 1-\frac{1}{|w|^2}\right) .
	\end{align*}
	Thus, one concludes for the result in Proposition \ref{BilnGenK2} by making use of the Meixner bilinear relation in  \cite[Eq. (12), p. 85]{Erdelyi1953}. 
\end{proof}

%
%


\subsection{Integral representations}
By means of the classical integral representations for the Gauss hypergeometric functions in the right hand side of \eqref{Zhyp2} and \eqref{ZhypF} or for the Jacobi polynomials in  \eqref{ZJacobi}, we can derive different integral representations for $\mathcal{Z}_{m,n}^{\rho,\kappa} (z,\bar z)$. 
However, we give below some non-trivial ones.
The first one is based on the Cauchy integral formula for holomorphic functions and following in spirit Kazantsev and Bukhgeimwe idea  \cite{KazantsevBukhgeim2004}. 

\begin{theorem}
	The fractional Zernike functions $\mathcal{Z}^{\kappa,\rho}_{m,n}$ admits  the following integral representation
	\begin{align} \label{IntRepbZ}
		\mathcal{Z}^{\kappa,\rho}_{m,n}(z,\bz)&= \dfrac{ (-1)^m m!}{2\pi i}  z^{-\rho}  (1-|z|^2)^{-\kappa}  	\oint_{\mid t\mid=1} t^{n+m+\rho+\kappa}\dfrac{        
			(\overline{t}-\bz)^{\kappa+m} } {(t-z)^{m+1}} dt .
	\end{align}
\end{theorem}

\begin{proof} Make use of the ordinary binomial expansion with the factorial function \cite[p. 56]{Rainvilla1971},
	$$\displaystyle (1- \xi)^{-a}= \sum\limits_{j=0}^{+\infty}{(a)_{j}} \frac{\xi^{j}}{j!},$$
	to expand the factor $(1-|z|^2)^j$ in the explicit expression of $\mathcal{Z}^{\kappa_m,\rho}_{m,n} $ given by 
	\eqref{Zmnexp}. Also, we need to the fact that
	$$\dfrac{\partial^{m}}{\partial z^{m}}(z^{j+n+\rho})=\dfrac{m!}{2\pi i}\oint_{\mid t\mid=1}\dfrac{t^{j+n+\rho}}{(t-z)^{m+1}}dt,$$
	which follows from the Cauchy integral formula applied to the function $\varphi_z(t)= t^{m+j+\rho}/(t-z)$. Thus, we obtain
	\begin{align*}
		\mathcal{Z}^{\kappa,\rho}_{m,n}(z,\bz) 
		&= (-1)^m z^{-\rho}  (1-|z|^2)^{-\kappa} 
		\sum\limits_{j=0}^{+\infty}\dfrac{(-\kappa-m)_{j}}{j!} \bz^{j} \frac{\partial^m}{\partial z^m}  
		\left( z^{n+\rho+j}\right) \\
		&= \dfrac{ (-1)^m m!}{2\pi i}  z^{-\rho}  (1-|z|^2)^{-\kappa}  
		\oint_{\mid t\mid=1}\dfrac{ t^{n+\rho}}{(t-z)^{m+1}} 
		\left( \sum\limits_{j=0}^{+\infty}(-\kappa-m)_{j} \dfrac{(t\bz)^{j}}{j!}\right)    dt \\
		&= \dfrac{ (-1)^m m!}{2\pi i}  z^{-\rho}  (1-|z|^2)^{-\kappa}  	\oint_{\mid t\mid=1}\dfrac{ t^{n+\rho} (1-t\bz)^{\kappa+m}}{(t-z)^{m+1}} dt \\
		&= \dfrac{ (-1)^m m!}{2\pi i}  z^{-\rho}  (1-|z|^2)^{-\kappa}  	\oint_{\mid t\mid=1} t^{n+m+\rho+\kappa}\dfrac{        
			(\overline{t}-\bz)^{\kappa+m} } {(t-z)^{m+1}} dt   
		.
	\end{align*}	
	This proves \eqref{IntRepbZ}.
\end{proof}

The next integral representation for the fractional Zernike functions appears as corollary of the bilinear generating function 
in Proposition \ref{BilnGenK2}. 	
		
\begin{proposition} Let $\gamma^{\kappa,\rho}_{m,n} $ be as in \eqref{normFrZ}. The fractional Zernike functions have the integral representation
	\begin{align} \label{IntRep2}
		\mathcal{Z}^{\kappa,\rho}_{m,n}(w,\bw) 
		=
		 \frac{m! (\kappa+1)_m \gamma^{\kappa,\rho}_{m,n} } { w^{\rho}t^m  }  &\int_{D}    	
		 \Xi^{\kappa,\rho}_{m,n}(z,w|t)
		 {_2F_1}\left( \begin{array}{c} -m , n+ \kappa+\rho+1 \\ \kappa+1 \end{array}\bigg | 1- |z|^2 \right) 
		 	\\& 
		  	\times  
		 	{_2F_1}\left(  \begin{array}{c}-n-\rho       ,-n -\rho \\ \kappa+1 \end{array}\bigg | - \frac{t(1-|z|^2)(1-|w|^2)}{(w+tz)(\bz+t\bw)}\right) 	d\lambda(z); \nonumber
	\end{align}
	where 
	$$ \Xi^{\kappa,\rho}_{m,n}(z,w|t):=\frac{ \bz^{m-n-\rho}  (w+t z)^{n+\rho} (\bz+t \bw)^{n+\rho}  }{ (1+t z \bw)^{2(n+\rho)+\kappa+1}} (1-|z|^2)^{\kappa}.$$
\end{proposition}

	\begin{proof}		
	 Starting from the expansion \eqref{expHH} one gets  
		$$\mathcal{Z}^{\kappa,\rho}_{m,n}(w,\bw)  =   
		\frac{m! (\kappa+1)_m \gamma^{\kappa,\rho}_{m,n} } { t^m  } \overline{\scal{	G^{\kappa,\rho}_n(\cdot,w|t)   , \mathcal{Z}^{\kappa,\rho}_{m,n}}_{L^{2,\kappa}_{\rho}(\D)}}
		.$$
		Next, by the  explicit expression of the kernel function $G^{\kappa,\rho}_n(z,w|t) $ given in Proposition \ref{BilnGenK2} combined with the hypergeometric representation in \eqref{Zhyp2} we derive the integral representation \eqref{IntRep2}.
	\end{proof}

\begin{remark} 	The integral in the right hand side of \eqref{IntRep2} is a rigid integral on complex domain $\D$ in the sense that it is nontrivial and can not be reduced to classical integral on real domains.  
\end{remark}

\subsection{Completeness.}

Here we discuss the completeness of the fractional Zernike functions in $L^{2,\kappa}_{\rho}(\D)$. 
Notice for instance that for $\rho=0,1,2, \cdots$ the functions $z^\rho\mathcal{Z}^{\kappa,\rho}_{m,n}$ for varying $m,n+\rho=0,1,2,\cdots$ constitute an orthogonal basis of $L^{2,\kappa}_{\rho}(\D)$ since they are closely connected to the complex Zernike polynomials  $\mathcal{Z}^{\kappa}_{m,n+\rho}$ 
by \eqref{frZcZ}. 
The latter ones are known  to form an orthogonal basis for the Hilbert space  $L^{2,\kappa}_{0}(\D)$ (see e.g., \cite{Dunkl84,Kanjin2013}).
 A direct proof starts from the observation that each $(\mathcal{Z}^{\kappa,\ell}_{m,n})_{m,n}$ is a polynomial of exact degrees $m$ in $\bz$ and $n$ in $z$, so that 
any $z^n \bz^m$ can be rewritten as   
$$\bz^m z^n = \sum_{j= 0}^m\sum_{k=0}^n a_{j,k}^{m,n} \mathcal{Z}^{\kappa,\ell}_{j,k}(z,\bz) .$$
The coefficients $a_{j,k}^{m,n}$ are explicit and can be computed by the formula
\begin{align*}
	a_{j,k}^{m,n} & =  \gamma^{\kappa,\rho}_{m,n} \int_{\D} \bz^m z^n \overline{ \mathcal{Z}^{\kappa,\ell}_{j,k}}(z,\bz) |z|^{2\rho}(1-|z|^2)^\kappa  d\lambda(z) ,
\end{align*}	
where $\gamma^{\kappa,\rho}_{m,n} $ is as in \eqref{normFrZ}. 
In the sequel, we consider only the case of $\rho$ non-integer $\rho>-1$.

\begin{theorem}\label{Mth5} 
	The functions $ \varUpsilon_{m,s}^{\kappa,\rho} :=    \bz^{-s/2}\mathcal{Z}_{m,m+s/2}^{\kappa,\rho-s/2}$,
	for varying nonnegative integer $m$ and varying integer $s$, form an orthogonal complete system  in $L^{2,\kappa}_{\rho}(\D)$.
\end{theorem} 

\begin{proof}
The key observation is that for $z= \sqrt{(1+x)/2}\, e^{i\theta} $ with $x\in[-1,1[$ and $\theta\in [0,2\pi[$ we have
	$$ \varUpsilon_{m,s}^{\kappa,\rho}(z,\bz) = 
	\frac{z^s}{|z|^s} \mathcal{Z}_{m,m}^{\kappa,\rho} (z,\bz) = m! e^{i s\theta} P_m^{(\kappa,\rho)}(x) .$$
	Subsequently,  their orthogonality in $L^{2,\kappa}_{\rho}(\D)$ readily follows using the orthogonality of the Jacobi polynomials $P_m^{(\kappa,\rho)}$ in the Hilbert $L^2_{\kappa,\rho} $ of square integrable functions on $[-1,1[$ with respect to the measure 
	 $
	 (1-x)^{\kappa} (1+x)^{\rho} dx$. More exactly we have    
\begin{align*}
		\int_{D} \varUpsilon_{m,s}^{\kappa,\rho}(z,\bz) 
		\overline{\varUpsilon_{n,r}^{\kappa,\rho}(z,\bz)}  |z|^{2\rho}(1-|z|^2)^\kappa  d\lambda(z)
		& = \frac{m!n!\pi }{2^{\kappa+\rho+1}} \norm{ P_m^{(\kappa,\rho)}}_{L^2_{\kappa,\rho} }^2 \delta_{m,n} \delta_{s,r}.
	\end{align*}
For their completeness, let 
$f\in F^{\perp}$, the orthogonal of $F =Span\{\varUpsilon_{m,s}^{\kappa,\rho}; m=0,1,2,\cdots , s\in \Z \} $ in $L^{2,\kappa}_{\rho}(\D)$.
%
	Thus, by Proposition \ref{proppsiJac},
 the assumption that $f\in F^{\perp}$  becomes equivalent to 
	\begin{align*}
		\scal{f, \varUpsilon_{m,s}^{\kappa,\rho} }_{L^{2,\kappa}_{\rho}} 
		=	& \frac{m!}{2^{\kappa+\rho+2}}  \int_{-1}^{1}
		\left(1-x\right) ^{\kappa/2} \left(1+x\right) ^{\rho/2}  P_m^{(\kappa,\rho)}(x) \hat{f}^{\kappa,\rho}_{s}(x)  dx = 0  
	\end{align*}
for every integer $s$ and $m=0,1,2,\cdots$.  The involved 
 function is defined by 
$$ \hat{f}^{\kappa,\rho}_{s}(x) :=\left(1-x\right) ^{\kappa/2}\left(1+x\right) ^{\rho/2 } \hat{f}_{s}(x) ,$$ where
$\hat{f}_{s}(x) $
denotes the $s$-th Fourier coefficient of  the function   
$  f_x :=
\theta \longmapsto f(\sqrt{(1+x)/2}\, e^{i\theta} ) 
$
for every fixed $x\in [-1,1[$.
Clearly $ \hat{f}^{\kappa,\rho}_{s}$
belongs to $L^{2}( [ -1,1[;dt)$ since 
by means of the Cauchy-Schwartz inequality and the Fubini's theorem one gets 
\begin{align*}
	\int_{-1}^1 |\hat{f}^{\kappa,\rho}_{s}(x)|^2   dx	
	& \leq 2 \pi \int_{-1}^1 \left(1-x\right)^{\kappa}\left(1+x\right)^{\rho}   \left(\int_0^{2\pi} \left|    f\left( \sqrt{\frac{1+x}{2}} e^{i\theta} \right) \right|^2 d \theta\right)  dx 
	= 2^{\kappa+\rho+3}\pi  \norm{f}_{L^{2,\kappa}_{\rho}}^2.
\end{align*}
%
%
Therefore,  $\hat{f}^{\kappa,\rho}_{s}=0$ 
 a.e on $[ -1,1[$ for every $s \in \Z$ for the  functions $   \left(1-x\right) ^{ \kappa/2}  \left(1+x\right) ^{\rho/2} P_m^{(\kappa,\rho)} $,  $m=0,1,2,\cdots$, being an orthogonal basis of $L^{2}\left( [ -1,1[,dt\right)$.
This implies in particular that the Fourier transform of $f_x    \in L^{2}([ 0,2\pi[,d\theta)$ satisfies 
$\mathcal{F}(f_x) (\ell) = \hat{f}_{-s}(x) = 0$ for every $s\in \Z$ and every fixed $x\in [-1,1[\setminus N$, where
we have set  $N:= \cup_s \{ x \in [-1,1 [ ; \, \hat{f}_{s} (x)\neq 0 \}$. 
	This proves that the function $f_{x}=0$ a.e. on $[ 0,2\pi[$ for almost every $x \in [ -1,1[$. Therefore, $f$ 	is a vanishing function almost every $t \in  [ -1,1[$.
	This completes the proof.
\end{proof}

\begin{corollary}
	The Hilbert spaces $A^{\kappa,\rho}_s(\D) :=  \overline{Span\{  \bz^{-s/2}\mathcal{Z}_{m,m+s/2}^{\kappa,\rho-s/2}; m=0,1,2,\cdots  \}}^{L^{2,\kappa}_{\rho}}; \, \, s \in \Z$
	defines a Hilbertian orthogonal decomposition of $L^{2,\kappa}_{\rho}(\D)$. Namely, we have 
	$$L^{2,\kappa}_{\rho}(\D) = \bigoplus_{s\in \Z} A^{\kappa,\rho}_s(\D) .$$ 
\end{corollary}

\begin{definition}
	The  closed subspace  $A^{\kappa,\rho}_s(\D)$  are called generalized (poly-meromorphic) Bergman  spaces of second kind.  
\end{definition}

\noindent{\bf Acknowledgement:}
The assistance of the members of  "Ahmed Intissar" and "Analysis, P.D.E. $\&$ Spectral Geometry" seminars is gratefully acknowledged.

\quad 

\noindent {\bf Data availability statement:}	All data generated or analyzed during this study are included in this article.

\quad 

\noindent {\bf Conflict of interest:}
The authors declare that they have no conflict of interest.
.

\end{document}